\documentclass[final,leqno,letterpaper]{etna}

\usepackage{graphicx}
\usepackage{subfigure}
\usepackage{amsmath}
\usepackage{amsfonts}
\usepackage{hyperref,microtype}
\usepackage{placeins}

\usepackage{lineno}

\newtheorem{thm}{Theorem}[section]
\newtheorem{cor}[thm]{Corollary}
\newtheorem{lem}[thm]{Lemma}

\setbibdata{1}{xx}{46}{2023} 

\hypersetup{%
    pdftitle={Convergence Analysis of a Krylov Subspace Spectral Method for the 1-D Wave Equation in an Inhomogeneous Medium},
    pdfauthor={Bailey Rester, James V. Lambers},
    pdfkeywords={spectral methods, wave equation, convergence analysis, variable coefficients}
    }

\title{Convergence Analysis of a Krylov Subspace Spectral Method for the 1-D Wave Equation in an Inhomogeneous Medium\thanks{%
Received... Accepted... Published online on... Recommended by....
}}

\author{Bailey Rester\footnotemark[2]
\and Anzhelika Vasilyeva\footnotemark[2]
        \and James V. Lambers\footnotemark[2]}

\shorttitle{Convergence Analysis of KSS for the 1-D Wave Equation} 
\shortauthor{B.~RESTER ET AL.}

\begin{document}

\maketitle

\renewcommand{\thefootnote}{\fnsymbol{footnote}}

\footnotetext[2]{School of Mathematics and Natural Sciences, The University of Southern Mississippi, 118 College Dr \#5043, Hattiesburg, MS 39406 USA}

\begin{abstract}
This paper presents a convergence analysis of a Krylov subspace spectral (KSS) method applied to a 1-D wave equation in an inhomogeneous medium.  It will be shown that for sufficiently
regular initial data, this KSS method yields unconditional stability,
spectral accuracy in space, and second-order accuracy in time,
in the case of constant wave speed and a bandlimited reaction term coefficient.
Numerical experiments that corroborate the established theory are included, along with an investigation of generalizations, such as to higher space dimensions and nonlinear PDEs, that features performance
comparisons with other Krylov subspace-based time-stepping methods.
This paper also includes the first stability analysis of a KSS method that does not assume 
a bandlimited reaction term coefficient.
\end{abstract}



\section{Introduction}


Consider the 1-D wave equation in an inhomogeneous medium,
\begin{equation} \label{eq:introPDE}
u_{tt}=(p(x)u_x)_x+q(x)u,
\end{equation}
on a bounded domain, with appropriate initial and boundary conditions.
%
Analytical methods are not practical to use for this problem, since the coefficients are not constant.  For instance, applying separation of variables  \cite{517txtbook} would result in a spatial ODE that cannot be solved analytically, and therefore numerical methods are needed.
However, standard time-stepping methods, such as Runge-Kutta methods or multistep methods, suffer from a lack of scalability. As the number of grid points increases, a smaller time step would be needed due to the CFL condition \cite{pdetxt} for explicit methods, or an increasingly ill-conditioned system must be solved for implicit methods. It follows that increasing the number of grid points significantly increases the computational expense. Therefore, a more practical numerical method for solving this kind of variable-coefficient PDE is desirable.

Krylov subspace spectral (KSS) methods are high-order accurate, explicit time-stepping methods that  possess stability characteristic of implicit methods \cite{paper2}. By contrast with other time-stepping methods, KSS methods employ a componentwise approach, in which each Fourier coefficient of the solution is computed using an approximation of the solution operator of the PDE that is tailored to that coefficient.  This customization is based on techniques for approximating bilinear forms involving matrix functions by treating them as Riemann-Stieltjes integrals \cite{mmq}. This componentwise approach allows KSS methods to circumvent difficulties caused by stiffness, and thus scale effectively to higher spatial resolution \cite{kssepi}.

A first-order KSS method applied to the heat equation with a constant leading coefficient was proven to be unconditionally stable \cite{paper22, paper18}, as well as a second-order KSS method applied to the wave equation with a constant leading coefficient \cite{paper17}.  In all of these studies, lower-order coefficients of the spatial differential operator were assumed to be bandlimited. A first-order KSS method applied to the heat equation with a bandlimited leading coefficient is also unconditionally stable \cite{paper2}. In this paper, we analyze stability of a KSS method applied to the wave equation with bandlimited coefficients. 

The outline of the paper is as follows. Section 2 provides an overview of KSS methods, as applied
to the wave equation. Section 3 presents a stability analysis of a second-order KSS method applied to the PDE (\ref{eq:introPDE}) with bandlimited coefficients $p(x)$ and $q(x)$, and periodic boundary conditions. In that same section, a full convergence analysis in the case of $p(x)\equiv \mathrm{constant}$ is carried out.  Corroborating numerical experiments are given in Section 4, along with application of 
the second-order KSS method to more general problems.  
This section also includes performance comparisons between the KSS method and other
time-stepping methods, particularly those that also make use of Krylov subspaces.
Upon demonstrating through numerical experiments 
that the assumptions on the coefficients of the PDE made in Section
3 are not necessary for convergence, an additional stability analysis is conducted in which
$q(x)$ is not assumed to be bandlimited, which has not previously been performed on a KSS method.
Conclusions and ideas for future work are given in Section 5.

%

\section{Background}

Consider the second-order wave equation 
\begin{equation} \label{pde}
u_{tt}+Lu=0 \textnormal{  on  } (0,2\pi)\times(0,\infty),
\end{equation} 
\begin{equation} \label{ICs}
u(x,0)=f(x), \hspace{1cm} u_{t}(x,0)=g(x), \hspace{1cm} 0<x<2\pi,
\end{equation} 
with periodic boundary conditions 
\begin{equation} \label{BCs}
u(0,t)=u(2\pi,t), \hspace{1cm} t>0.
\end{equation}
The spatial differential operator $L$ is defined by 
\begin{equation} \label{eqLform}
Lu=-(p(x)u_x)_x+q(x)u,
\end{equation}
where we assume $p(x)>0$ and $q(x)\geq0$, to guarantee that $L$ is self-adjoint and positive definite. 

A spectral representation of the operator $L$ allows us to describe the solution operator, the propagator,
as a function of $L$ \cite{glk}. 
By introducing
\begin{eqnarray} 
f_{11}(\lambda)&=&f_{22}(\lambda)=\cos(\lambda^{1/2} \Delta t), \label{eqf11} \\
f_{12}(\lambda)&=&\lambda^{-1/2}\sin(\lambda^{1/2}\Delta t), \label{eqf12} \\
f_{21}(\lambda) &=& -\lambda f_{12}(\lambda), \label{eqf21}
\end{eqnarray}
%
we can describe the evolution of the solution by
$$
\left[ \begin{array}{c} u(x,t+\Delta t) \\ u_t(x,t+\Delta t) \end{array}\right]=\left[ \begin{array}{cc} f_{11}(L) &f_{12}(L)\\ f_{21}(L) &f_{22}(L) \end{array}\right]\left[ \begin{array}{c} u(x,t) \\ u_t(x,t) \end{array}\right].
$$
In view of the periodic boundary conditions, we can express the solution at time $t+\Delta t$
as a sum of Fourier series,
\begin{eqnarray*}
u(x,t+\Delta t) &=& \frac{1}{2\pi} \sum_{\omega=-\infty}^\infty e^{i\omega x} \langle e^{i\omega\cdot},
\cos(L^{1/2}\Delta t) u(\cdot,t) \rangle + \\
&& \frac{1}{2\pi} \sum_{\omega=-\infty}^\infty e^{i\omega x} \langle e^{i\omega\cdot},
L^{-1/2}\sin(L^{1/2}\Delta t) u_t(\cdot,t) \rangle,
\end{eqnarray*}
where $\langle \cdot, \cdot \rangle$ is the standard inner product of functions on $(0,2\pi)$.

Upon spatial discretization, each Fourier coefficient in the above series is approximated by
an expression of the form
\begin{equation} \label{eq:ufAv}
\mathbf{u}^Hf(A)\mathbf{v},
\end{equation}
where ${\bf u}$ and ${\bf v}$ are $N$-vectors, $A$ is an $N\times N$ symmetric positive definite matrix, 
and $f$ is either $f_{11}$ or $f_{12}$.  In \cite{mmq}, Golub and Meurant describe algorithms
for approximating such bilinear forms involving matrix functions, by treating them as Riemann-Stieltjes
integrals that can be approximated through Gauss quadrature over an interval containing the eigenvalues
of $A$.

In the case of ${\bf u}={\bf v}$, the Gauss quadrature rule is constructed by applying the Lanczos algorithm to $A$, with initial vector ${\bf u}$.  The Gauss quadrature nodes and weights are then obtained from
the eigenvalues and eigenvectors of the tridiagonal matrix of recursion coefficients produced by the
Lanczos iteration.  In the case of ${\bf u}\neq{\bf v}$, the unsymmetric Lanczos algorithm can be used
instead, with initial vectors ${\bf u}$ and ${\bf v}$, but this may yield a quadrature rule that does not have
real positive weights, which can be numerically unstable \cite{Atk}.

For this case, a block approach can be used instead \cite{mmq}.  In this case, the block Lanczos
algorithm \cite{gu} is applied to $A$, with initial block 
$\left[ \begin{array}{cc} {\bf u} & {\bf v} \end{array} \right]$.  The iteration produces a block tridiagonal
matrix, with $2\times 2$ blocks, and as before, its eigenvalues and eigenvectors yield Gauss
quadrature nodes and (matrix-valued) weights.


We now describe how this approach is applied to the solution of the problem
(\ref{pde}), (\ref{ICs}), (\ref{BCs}).
Let ${\bf u}^n$ and ${\bf u}_t^n$ be the computed solution at time $t_n$ and its
time derivative, respectively, and let $\hat{\bf e}_\omega$ be a discretization of
$\hat{e}_\omega(x)=e^{i\omega x}$.
For each wave number
$\omega = -N/2+1,\ldots, N/2$, we define 
$$R_0 = \left[ \begin{array}{cc} \frac{1}{N}\hat{\bf e}_\omega & {\bf u}^n \end{array} \right],
\quad \tilde{R}_0 = \left[ \begin{array}{cc} \frac{1}{N}\hat{\bf e}_\omega & {\bf u}_t^n \end{array} \right],$$
and then compute the $QR$ factorizations
$$R_0 = X_1B_0, \quad \tilde{R}_0 = \tilde{X}_1 \tilde{B}_0.$$
Block Lanczos iteration yields ${\cal T}_K$ and $\tilde{\cal T}_K$ from $X_1$ and $\tilde{X}_1$.
Then, the Fourier coefficients of the solution and its time derivative are approximated by
$$[\hat{\bf u}^{n+1}]_\omega = \left[ B_0^H \cos[{\cal T}_K^{1/2}\Delta t]_{1:2,1:2} B_0 \right]_{12} + \left[ \tilde{B}_0^H (\tilde{\cal T}_K^{-1/2} \sin[\tilde{\cal T}_K^{1/2}\Delta t])_{1:2,1:2} \tilde{B}_0 \right]_{12},$$
$$[\hat{\bf u}_t^{n+1}]_\omega = -\left[ B_0^H ({\cal T}_K^{1/2}\sin[{\cal T}_K^{1/2}\Delta t])_{1:2,1:2} B_0 \right]_{12} + \left[ \tilde{B}_0^H \cos[\tilde{\cal T}_K^{1/2}\Delta t]_{1:2,1:2} \tilde{B}_0 \right]_{12}.$$

Let $u(x,\Delta t)$ be the exact solution,
and let $\tilde{u}(x,\Delta t)$ be the approximate solution.  If $K$ iterations of block Lanczos are performed, then, for $\omega=-N/2+1,\ldots,N/2$,  \cite{paper17}
$$
| \langle \hat{ e}_\omega, u(\cdot,\Delta t) - \tilde{u}(\cdot,\Delta t) \rangle | = O(\Delta t^{4K}),
$$
$$
| \langle \hat{ e}_\omega, u_t(\cdot,\Delta t) - \tilde{u}_t(\cdot,\Delta t) \rangle | = O(\Delta t^{4K-1}).  
$$
The high order of accuracy in time is due to the second derivative with respect to time in the PDE.
In addition to their high-order accuracy in time, 
the following has been proven about the stability of KSS methods, for various problems:
\begin{itemize}
\item Heat equation $u_t=pu_{xx}+q(x)u$, where $p$ is constant, $q(x)$ is bandlimited: a first-order KSS method is unconditionally stable \cite{paper18},
\item Wave equation $u_{tt}=pu_{xx}+q(x)u$, where $p$ is constant, $q(x)$ is bandlimited: a second-order KSS method is unconditionally stable \cite{paper17},
\item Reaction-diffusion system of the form ${\bf v}_t=L{\bf v}$: a first-order KSS method with constant diffusion coefficient and bandlimited reaction term coefficient is unconditionally stable \cite{paper2},

\item Wave equation $u_{tt}=pu_{xx}+q(x)u$, where $p$ is constant, $q(x)$ is bandlimited: a second-order block KSS method is unconditionally stable \cite{paper17}, and
\item Heat equation $u_t=(p(x)u_x)_x+q(x)u$ where $p(x)$ and $q(x)$ are bandlimited: a first-order block KSS method is unconditionally stable \cite{paper2}.
\end{itemize}

KSS methods use a significantly different approach to computing matrix function-vector products
of the form $\varphi(A){\bf b}$ than Krylov subspace methods from the literature (see,
for example, \cite{hochlubsel}).  Such Krylov subspace methods approximate the function $\varphi$
with either a polynomial or rational function. Depending on the function $\varphi$, the approximating function may need to be of high degree to ensure sufficient accuracy.  When such methods are used
to solve stiff systems of ODEs obtained from spatial discretization of PDEs, the degree can grow substantially when the time step or number of grid points increases.  

This is demonstrated in
\cite{kssepi}, where it was also shown that, by contrast, KSS methods do not suffer from this loss of
scalability.  Each Fourier coefficient of the solution is obtained using
its own frequency-dependent approximation, that is of a low degree determined by the desired order
of temporal accuracy.  This is possible because each Fourier coefficient is equivalent to a 
Riemann-Stieltjes integral with a frequency-dependent measure that is nearly constant over most
of the domain of integration \cite{kssepi}, and therefore the integral is determined primarily by
the behavior of the integrand over only a small, frequency-dependent portion of this domain.

In Section 4 it will be demonstrated that this component-wise approach to time-stepping provides
an advantage over other time-stepping methods, that apply the same approximation of the
exponential to all components of the solution.

%



\section{Convergence Analysis} \label{secconv}

We will now analyze convergence of a second-order KSS method, with $K=1$,
for the IVP (\ref{pde}), (\ref{ICs}), (\ref{BCs}), (\ref{eqLform}),
under the assumptions that the Fourier coefficients $\hat{p}(\omega)$, $\hat{q}(\omega)$ of
$p(x)$ and $q(x)$, respectively, satisfy $\hat p(\omega)=q(\omega)=0$ when $|\omega|>\omega_{\textnormal{max}}$ for some threshold $\omega_{\textnormal{max}}$.
That is, we assume that $p(x)$ and $q(x)$ are bandlimited.  

We first carry out spatial discretization.  We use a uniform grid, with spacing
$\Delta x = 2\pi/N$, where $N$ is assumed to be even.  
Then, we let ${\bf x}_N$ be an $N$-vector of grid points
$$x_j = j\Delta x,  \quad j=0,1,2,\ldots,N-1.$$
We denote by $\omega_j$ the corresponding wave numbers
$$\omega_j = j-N/2, \quad j=1,2,\ldots,N.$$ 
We then denote by $D_N$ an $N\times N$ matrix that discretizes the second derivative operator
using the discrete Fourier transform:
$$D_N = F_N^{-1} \Lambda_N F_N,$$
where
$$[F_N]_{jk} = \frac{1}{N} e^{-i\omega_j x_k}, \quad
[\Lambda_N]_{jj} = -\omega_j^2.$$
We also let $I_N$ denote the $N\times N$ identity matrix, whereas $I$ is the identity operator
on functions of $x$.

Let $u_1=u$ and $u_2=u_t$. We then rewrite (\ref{pde}) as the first-order system 
\begin{eqnarray}
\frac{\partial u_1}{\partial t} &=& u_2, \label{eqpdesys1} \\
\frac{\partial u_2}{\partial t} &=& -Lu_1, \label{eqpdesys2}
\end{eqnarray}
which, for convenience, we write as
\begin{equation} \label{eqpdesysall}
{\bf v}_t = \tilde{L}{\bf v}, \quad {\bf v}=\left[ \begin{array}{c} u_1 \\ u_2 \end{array} \right], \quad
\tilde{L} = \left[ \begin{array}{cc} 0 & I \\ -L & 0 \end{array} \right].
\end{equation}
Spatial discretization of (\ref{eqpdesysall}) yields a system of ODEs
\begin{equation} \label{eqwaveodesys}
{\bf v}_N'(t) = \tilde{L}_N {\bf v}(t),
\end{equation}
where 
$${\bf v}_N(t) = \left[ \begin{array}{c}
{\bf u}_{1,N}(t) \\
{\bf u}_{2,N}(t)
\end{array} \right]$$
is the spatial discretization of the vector field ${\bf v}$, and 
$\tilde{L}_N$ is a $2N\times 2N$ matrix which has the $2\times 2$ block structure
$$\tilde{L}_N = \left[ \begin{array}{cc} 0 & I_N \\ -L_N & 0 \end{array} \right].$$

We define the exact solution operator of (\ref{eqpdesys1}), (\ref{eqpdesys2}) as 
\begin{equation} 
S(t)=\exp[\tilde{L}t] = \left[ \begin{array}{cc} S_{11}(t) & S_{12}(t) \\ S_{12}(t) & S_{22}(t) \end{array}\right]=\left[ \begin{array}{cc} R_0(t) &R_1(t)\\ -L\, R_1(t)&R_0(t)\end{array}\right], \label{P} \end{equation} where, as before, $R_1(t) = L^{-1/2}\sin(L^{1/2} t)$ and $R_0(t) = \cos(L^{1/2} t)$.
Then we let  
\begin{equation} \label{eqSNdef}
S_N(\Delta t)=\left[ \begin{array}{cc} {S}_{N,11}(\Delta t) & {S}_{N,12}(\Delta t)\\ {S}_{N,12}(\Delta t) & {S}_{N,22}(\Delta t)\end{array}\right],
\end{equation} 
where each ${S}_{N,ij}(\Delta t)$ is the approximation of $S_{ij}(\Delta t)$ by the KSS method. 

A KSS method applied to (\ref{pde}) with $K=1$ uses two block Gauss quadrature nodes for each Fourier
coefficient.  Using an approach described in \cite{paper8}, we estimate these nodes, rather than using block Lanczos
iteration explicitly.  This significantly improves the efficiency of KSS methods, but it will
also simplify the convergence analysis to be carried out in this section.  The quadrature nodes will be prescribed as follows:
\begin{equation} \label{eqinterpw}
l_{1,\omega} = 0, \quad l_{2,\omega} = \overline{p}\omega^2 + \overline{q},
\quad \omega=-N/2+1,\ldots,N/2,
\end{equation}
where $\overline{p}$ and $\overline{q}$ are the average values of $p(x)$ and $q(x)$, respectively, on $[0,2\pi]$.

To interpolate the functions $f_{ij}(\lambda)$ from (\ref{eqf11}), (\ref{eqf12}), (\ref{eqf21})
at the nodes $l_{1,\omega}$, $l_{2,\omega}$, we compute the slopes
$$
M_{ij,\omega} = \frac{f_{ij}(l_{2,\omega})-f_{ij}(l_{1,\omega})}{l_{2,\omega}-l_{1,\omega}}, 
\quad \omega=-N/2+1,\ldots,N/2,
\quad i,j=1,2.
$$
We then describe the computed solution at time $t_{n+1}$ by
\begin{eqnarray}
{\bf u}^{n+1} & = & {\bf z}_{11}+ {\bf z}_{12}, \nonumber \\
{\bf u}_t^{n+1} & = & {\bf z}_{21} + {\bf z}_{22}, \label{eqzintro}
\end{eqnarray}
where 
$${\bf z}_{i1}={S}_{N,i1}(\Delta t){\bf u}^n, \quad 
{\bf z}_{i2}={S}_{N,i2}(\Delta t){\bf u}_t^n, \quad i=1,2.$$
We also define
$\tilde{p}=p-\bar{p}$, $\tilde{q}=q-\bar{q}$, and let
$P_N$, $Q_N$, $\tilde{P}_N$, and $\tilde{Q}_N$ be diagonal matrices with the values of the coefficients $p(x)$, $q(x)$,
$\tilde{p}(x)$, and $\tilde{q}(x)$, respectively, at the grid points on the main diagonal.  

Let $I_\omega = \{ k\in\mathbb{Z} | 0 < |k-\omega| \leq \omega_{\max} \}$.
The discrete Fourier coefficients of ${\bf z}_{ij}$, $i,j=1,2$, are then given by
\begin{eqnarray}
\hat{\bf z}_{11}(\omega)&=&S_{11}(l_{2,\omega})(\hat {\bf e}_\omega^H {\bf u}^n) + M_{11,\omega}\hat {\bf e}_\omega^H(L_N-l_{2,\omega}I){\bf u}^n \nonumber \\
&=& S_{11}(l_{2,\omega})\hat u(\omega) - i\omega M_{11,\omega}\sum_{k\in I_\omega} \hat p(\omega-k)i(k)\hat u(k) +  \nonumber \\
& & M_{11,\omega}\sum_{k\in I_\omega}\hat q(\omega-k)\hat u(k), \label{eqz11} \\
\hat{\bf z}_{21}(\omega)
&=& S_{21}(l_{2,\omega})\hat u(\omega) - i\omega M_{21,\omega}\sum_{k\in I_\omega} \hat p(\omega-k)i(k)\hat u(k) +  \nonumber \\
& & M_{21,\omega}\sum_{k\in I_\omega}\hat q(\omega-k)\hat u(k), \label{eqz21} \\
\hat{\bf z}_{12}(\omega) &=& S_{12}(l_{2,\omega}) \hat {u}_t(\omega)  - i\omega M_{12,\omega}\sum_{k\in I_\omega} \hat p(\omega-k)i(k)\hat {u}_t(k) +  \nonumber \\
& & M_{12,\omega}\sum_{k\in I_\omega}\hat q(\omega-k)\hat {u}_t(k), \label{eqz12} \\
 \hat{\bf z}_{22}(\omega) &=& S_{22}(l_{2,\omega}) \hat {u}_t(\omega) - i\omega M_{22,\omega}\sum_{k\in I_\omega} \hat p(\omega-k)i(k)\hat {u}_t(k) +  \nonumber \\
& & M_{22,\omega}\sum_{k\in I_\omega}\hat q(\omega-k)\hat {u}_t(k), \label{eqz22}
\end{eqnarray}
where $i=\sqrt{-1}$ and $\omega=-N/2+1,\ldots,N/2$.
To obtain these formulas, we used the simplification
\begin{eqnarray*}
\hat {\bf e}_\omega^H(L_N-l_{2,\omega}I){\bf u}^n 
&=& \hat {\bf e}_\omega^H [-D_N P_N D_N+Q_N]{\bf u}^n - l_{2,\omega}(\hat {\bf e}_\omega^H {\bf u}^n)\\
&=& -i\omega \hat {\bf e}_\omega^HP_N D_N{\bf u}^n + \hat {\bf e}_\omega^H Q_N {\bf u}^n - l_{2,\omega}(\hat {\bf e}_\omega^H {\bf u}^n)\\
&=& ( \bar{p}\omega^2 +\bar{q})(\hat {\bf e}_\omega^H{\bf u}^n) - i\omega\hat {\bf e}_\omega^H\tilde{P}_ND_N{\bf u}^n + \hat {\bf e}_\omega^H\tilde{Q}_N{\bf u}^n - l_{2,\omega}(\hat {\bf e}_\omega^H {\bf u}^n)\\
&=&  - i\omega\hat {\bf e}_\omega^H\tilde{P}_ND_N{\bf u}^n + \hat {\bf e}_\omega^H\tilde{Q}_N{\bf u}^n.
\end{eqnarray*}

To bound error, we need to establish an upper bound of a norm of the approximate solution operator $S_N(\Delta t)$. 
We elect to use the $C$-norm, defined by
$$\|\left(u,v\right)\|_C^2 = \langle u, Cu \rangle + \langle v, v \rangle,$$
where, as before, $\langle \cdot,\cdot\rangle$ is the standard inner product on $(0,2\pi)$, and its discrete counterpart, the
$C_N$-norm, defined by 
$$\|\left({\bf u},{\bf v}\right)\|_{C_N}^2={\bf u}^TC_N{\bf u}+\|{\bf v}\|_2^2.$$ 
The $N\times N$ matrix $C_N$ discretizes the constant-coefficient differential operator $C=-\bar{p}\partial_{xx}+\bar{q}I$, where $u$ and $v$ are $N$-vectors. We choose to bound the $C_N$-norm of the solution operator for convenience, because the operator $C_N$ has a very simple expression in Fourier space due to its constant coefficients, which simplifies the analysis. 

\subsection{Stability}

We wish to express $\|S_N(\Delta t)\|_{C_N}$ as the 2-norm of some matrix, since that will be easier to bound.
We define $\|S_N(\Delta t)\|_{C_N}$ by
$$\|S_N(\Delta t)\|_{C_N}^2=\sup_{{\bf w}=({\bf u},{\bf v})\neq{\bf 0}}\dfrac{\|S_N(\Delta t){\bf w}\|_{C_N}^2}{\|{\bf w}\|_{C_N}^2}.$$ 
Then \begin{equation}\|S_N(\Delta t)\|_{C_N}^2=\sup_{({\bf u},{\bf v})\neq{\bf 0}} \dfrac{{\bf \tilde{u}}^TC_N{\bf \tilde{u}}+\|{\bf\tilde{v}}_2\|^2}{{\bf u}^TC_N{\bf u}+\|{\bf v}_2\|^2}, \label{cnormofP} \end{equation} where $\begin{bmatrix} {\bf \tilde{u}}\\ {\bf \tilde{v}} \end{bmatrix}=S_N(\Delta t)\begin{bmatrix} {\bf u}\\ {\bf v} \end{bmatrix}=S_N(\Delta t){\bf w}$.
In matrix form, we have
\begin{eqnarray*}
\|S_N(\Delta t)\|_{C_N}^2&=&\sup\dfrac{{\bf w}^TS_N(\Delta t)^T\tilde{C}_N S_N(\Delta t){\bf w}}{{\bf w}^T\tilde{C}_N{\bf w}}
=\sup\dfrac{{\bf w}^TS_N(\Delta t)^T\tilde{C}_N S_N(\Delta t){\bf w}}{(\tilde{C}_N^{1/2}{\bf w})^T(\tilde{C}_N^{1/2}{\bf w})},
\end{eqnarray*}
where $\tilde{C}_N=\begin{bmatrix} C_N & 0\\ 0 & I \\\end{bmatrix}$. Let ${\bf z}=\tilde{C}_N^{1/2}{\bf w}$. Then 
\begin{eqnarray*}
\|S_N(\Delta t)\|_{C_N}^2
&=&\sup\dfrac{{\bf z}^T(\tilde{C}_N^{1/2}S_N(\Delta t)\tilde{C}_N^{-1/2})^T(\tilde{C}_N^{1/2}S_N(\Delta t)\tilde{C}_N^{-1/2}){\bf z}}{{\bf z}^T{\bf z}}.
\end{eqnarray*}
Therefore, $$\|S_N(\Delta t)\|_{C_N}=\|B\|_2=\sqrt{\rho\left(B^TB\right)} \leq\sqrt{\|G\|_\infty},$$ where $B=\tilde{C}_N^{1/2}S_N(\Delta t)\tilde{C}_N^{-1/2}$, and
 \begin{eqnarray} 
 G&=&B^TB 
= \begin{bmatrix}
G_{11} & G_{12}\\
G_{21} & G_{22}
\end{bmatrix},
\label{G}
\end{eqnarray}
with
\begin{eqnarray}
G_{11} & = & C_N^{-1/2}{S}_{N,11}(\Delta t)^TC_N{S}_{N,11}(\Delta t)C_N^{-1/2} + \nonumber \\ && 
C_N^{-1/2}{S}_{N,21}(\Delta t)^T{S}_{N,21}(\Delta t)C_N^{-1/2}, \\
G_{12} & = & C_N^{-1/2}{S}_{N,11}(\Delta t)^TC_N{S}_{N,12}(\Delta t) + C_N^{-1/2}{S}_{N,21}(\Delta t)^T{S}_{N,22}(\Delta t),\\
G_{21} & = & {S}_{N,12}(\Delta t)^TC_N{S}_{N,11}(\Delta t)C_N^{-1/2} + {S}_{N,22}(\Delta t)^T {S}_{N,21}(\Delta t)C_N^{-1/2}, \\
G_{22} & = & {S}_{N,12}(\Delta t)^TC_N{S}_{N,12}(\Delta t) + {S}_{N,22}(\Delta t)^T {S}_{N,22}(\Delta t).
\end{eqnarray}
To obtain a bound for the $C_N$-norm of the overall approximate solution operator $S_N(\Delta t)$, we will proceed
by bounding $\|G\|_\infty$ through bounding $\|G_{ij}\|_\infty$ for $i,j=1,2$.

To bound the norm of each such block, we use expressions for the computed
solution ${\bf u}^{n+1}$, ${\bf u}_t^{n+1}$ at time $t_{n+1}$, in terms of
${\bf u}^n$ and ${\bf u}_t^{n}$.  
We begin with
\begin{eqnarray*} 
\Vert ({\bf u}^{n+1},\textbf{u}_t^{n+1})\Vert_{C_N}^2 & = & (\textbf{u}^{n+1})^TC_N(\textbf{u}^{n+1})  +  (\textbf{u}_t^{n+1})^T(\textbf{u}_t^{n+1})\\
&=&  [{\bf u}^n]^T \bar{G}_{11} {\bf u}^n + [{\bf u}^n]^T\bar{G}_{12}{\bf u}_t^n + [{\bf u}_t^n]^T\bar{G}_{21}{\bf u}^n + [{\bf u}_t^n]^T\bar{G}_{22}{\bf u}_t^n
\end{eqnarray*}
where
\begin{eqnarray*}
\bar{G}_{11} & = & {S}_{N,11}(\Delta t)^T C_N {S}_{N,11}(\Delta t) + {S}_{N,21}(\Delta t)^T {S}_{N,21}(\Delta t), \\
\bar{G}_{12} & = & {S}_{N,11}(\Delta t)^T C_N {S}_{N,12}(\Delta t) + {S}_{N,21}(\Delta t)^T {S}_{N,22}(\Delta t), \\
\bar{G}_{21} & = & {S}_{N,12}(\Delta t)^T C_N {S}_{N,12}(\Delta t) + {S}_{N,22}(\Delta t)^T {S}_{N,21}(\Delta t), \\
\bar{G}_{22} & = & {S}_{N,12}(\Delta t)^T C_N {S}_{N,12}(\Delta t) + {S}_{N,22}(\Delta t)^T {S}_{N,22}(\Delta t).
\end{eqnarray*}
We note that
$$
G_{11} = C_N^{-1/2} \bar{G}_{11} C_N^{-1/2}, \quad
G_{12} = C_N^{-1/2} \bar{G}_{12}, \quad
G_{21} = \bar{G}_{21} C_N^{-1/2}, \quad
G_{22} = \bar{G}_{22}.
$$
Therefore, we can proceed by bounding the entries of each $\bar{G}_{ij}$, for $i,j=1,2$.

\vspace{0.1in}
\begin{lem} \label{lem31}
Assume $\hat p(\omega)=0$ and $\hat q(\omega)=0$ for $|\omega|>\omega_{\max}$. Then the matrix $G_{11}$ defined in (\ref{G}) satisfies 
\begin{eqnarray}
\|G_{11}\|_\infty &\leq& 1  +  C_{11,p} \|\tilde{p}\|_{\infty} \Delta t^2 N^2 +  C_{11,q} \|\tilde{q}\|_{\infty} \Delta t^2 + C_{11,p^2} \|\tilde{p}\|_{\infty}^2 \Delta t^2 N^2 + \nonumber \\ 
& & C_{11,pq} \|\tilde{p}\|_{\infty} \|\tilde{q}\|_{\infty} \Delta t^2 + C_{11,q^2} \|\tilde{q}\|_{\infty}^2 \Delta t^2
\label{G11bound}
\end{eqnarray}
where the constants $C_{11,p}$, $C_{11,q}$, $C_{11,p^2}$, $C_{11,pq}$, and $C_{11,q^2}$
are independent of $N$ and $\Delta t$. 
\end{lem}
\begin{proof} Let $\hat{I}_N = \{ -N/2+1,\ldots,N/2 \}$.
From (\ref{eqzintro}) we have
\begin{eqnarray*}
 [{\bf u}^n]^T \bar{G}_{11} {\bf u}^n&=& {\bf z}_{11}^T C_N {\bf z}_{11} + {\bf z}_{21}^T {\bf z}_{21}\\
 &=&\sum_{\omega\in \hat{I}_N} \overline{\hat{\bf z}_{11}(\omega)} \hat{\bf z}_{11}(\omega)(\bar{p}\omega^2+\bar{q}) + \sum_{\omega\in\hat{I}_N} \overline{\hat{\bf z}_{21}(\omega)} \hat{\bf z}_{21}(\omega) \\
 & = & \sum_{j\in\hat{I}_N} \sum_{k\in\hat{I}_N} \hat u(-j) \hat u(k) \left[ \bar{A} + \bar{B} + \bar{C} + \bar{D} \right]_{jk} 
\end{eqnarray*}
where, by (\ref{eqz11}) and (\ref{eqz21}), we have, for $j,k\in \hat{I}_N$,
\begin{eqnarray*}
\bar{A}_{jj} &=& \left(S_{11}(l_{2,j})\right)^2 (\overline{p} j^2 +\overline{q}) + \left(S_{21}(l_{2,j})\right)^2  \\
&=& \cos^2\left(\sqrt{\overline{p} j^2 +\overline{q}}\Delta t\right) \left(\overline{p} j^2 +\overline{q}\right) + \left(-\left(\overline{p} j^2 +\overline{q}\right)^{1/2}\sin\left(\sqrt{\overline{p} j^2 +\overline{q}}\Delta t\right)\right)^2  \\
&=& \overline{p} j^2 +\overline{q},
\end{eqnarray*}
\begin{eqnarray*}
\bar{B}_{jk} &=&  -jk S_{11}(l_{2,k}) M_{11,k} \hat p(k-j)(\overline{p} k^2 +\overline{q}) + S_{11}(l_{2,k}) M_{11,k}  \hat q(k-j)(\overline{p} k^2 +\overline{q}) + \\
& &  -jk S_{21}(l_{2,k}) M_{21,k} \hat p(k-j) +  S_{21}(l_{2,k}) M_{21,k} \hat q(k-j), \quad j \neq k,
\end{eqnarray*}
\begin{eqnarray*}
\bar{C}_{jk} &=& -jk S_{11}(l_{2,j}) M_{11,j}  \hat p(k-j)(\overline{p} j^2 +\overline{q})   +    S_{11}(l_{2,j})  M_{11,j} \hat q(k-j) (\overline{p} j^2 +\overline{q}) + \\
& & -jk S_{21}(l_{2,j})M_{21,j}  \hat p(k-j)    +    S_{21}(l_{2,j})M_{21,j} \hat q(k-j), \quad j \neq k,
\end{eqnarray*}
and
\begin{eqnarray*}
\bar{D}_{jk} 
&=&-jk \sum_{\omega\in\hat{I}_N\setminus\{k,j\}} \omega^2 \hat p(\omega-k)\hat p(j-\omega) \left(M_{11,\omega}^2(\overline{p} \omega^2 +\overline{q}) + M_{21,\omega}^2\right)  + \\
& &  (-k) \sum_{\omega\in\hat{I}_N\setminus\{k,j\}} \omega  \hat p(\omega-k) \hat q(j-\omega) \left(M_{11,\omega}^2(\overline{p} \omega^2 +\overline{q}) + M_{21,\omega}^2\right)  + \\
& &  (-j) \sum_{\omega\in\hat{I}_N\setminus\{k,j\}} \omega  \hat q(\omega-k) \hat p(j-\omega) \left(M_{11,\omega}^2(\overline{p} \omega^2 +\overline{q}) + M_{21,\omega}^2\right)  + \\
& &   \sum_{\omega\in\hat{I}_N\setminus\{k,j\}} \hat q(\omega-k) \hat q(j-\omega) \left(M_{11,\omega}^2(\overline{p} \omega^2 +\overline{q}) + M_{21,\omega}^2\right),
\end{eqnarray*}
with $\bar{A}_{jk}=0$ for $j\neq k$, and $\bar{B}_{jj} = \bar{C}_{jj} = 0$ for $j\in\hat{I}_N$.

To obtain an upper bound for $\|G_{11}\|_{\infty}$, we use the following bounds on $S_{ij}(l_{2,\omega})$ and $M_{ij,\omega}$, which are the coefficients in the linear approximations of the various components of the solution operator:
\begin{eqnarray*}
\left|S_{11}(l_{2,\omega})\right| &\leq& \left| \cos\left(l_{2,\omega}^{1/2} \Delta t\right)\right| \leq 1,\\ 
\left|S_{21}(l_{2,\omega})\right| &\leq& \left| - l_{2,\omega}^{1/2}\sin\left( l_{2,\omega}^{1/2}\Delta t\right) \right| \leq \left| l_{2,\omega}^{1/2} \right| l_{2,\omega}^{1/2}\Delta t = l_{2,\omega}\Delta t,\\
\left| M_{11,\omega}\right| &\leq& \dfrac{\Delta t^2}{2}, \\
\left| M_{11,\omega}\right| &\leq& \dfrac{\Delta t}{(\overline{p} \omega^2 +\overline{q})^{1/2}}, \\ 
\left| M_{21,\omega}\right| &\leq& \Delta t .
\end{eqnarray*}
We have multiple bounds for $M_{11}$ so that different terms will have the same order of magnitude in terms of $N$ and $\Delta t$. Then, for $j\neq k$, we have
\begin{eqnarray*}
\left| \bar{B}_{jk} \right| 
&\leq& \left|  j k \hat p(k-j)(\overline{p} k^2 +\overline{q}) \dfrac{\Delta t^2}{2} \right| +   \left| \dfrac{\Delta t^2}{2} \hat q(k-j)(\overline{p} k^2 +\overline{q}) \right| + \\
& & \left|  j k \hat p(k-j) (\overline{p}k^2+\overline{q})\Delta t^2 \right|+\left|  \hat q(k-j) (\overline{p}k^2+\overline{q})\Delta t^2  \right| \\
&\leq& \frac{3}{2} {\Delta t^2 }(\overline{p} k^2 +\overline{q}) \left( \left|  j k \hat p(k-j)\right|+ \left| \hat q(k-j)  \right| \right),
\end{eqnarray*}
\begin{eqnarray*}
|\bar{C}_{jk}| &\leq& \left|  j k \hat p(k-j)(\overline{p} j^2 +\overline{q}) \dfrac{\Delta t^2}{2} \right| + \left|  \hat q(k-j) (\overline{p} j^2 +\overline{q})\dfrac{\Delta t^2}{2} \right| + \\
& & \left|  j k \hat p(k-j) (\overline{p}j^2+\overline{q})\Delta t^2\right|  + \left|   \hat q(k-j) (\overline{p} j^2 +\overline{q})\Delta t^2 \right| \\
&\leq& \frac{3}{2} {\Delta t^2 } (\overline{p} j^2 +\overline{q})\left( \left|  j k \hat p(k-j)\right| +  \left|\hat q(k-j)  \right|\right),
\end{eqnarray*}
\begin{eqnarray*}
|\bar{D}_{jk}| 
&\leq& \left| jk \sum_{\omega\in\hat{I}_N\setminus\{k,j\}} \omega^2 |\hat  p(\omega-k)\hat p(j-\omega) |\left( 2\Delta t^2 \right) + \right. \\
& & \left.  k \sum_{\omega\in\hat{I}_N\setminus\{k,j\}} \omega  |\hat p(\omega-k) \hat q(j-\omega) |\left(2\Delta t^2\right) + \right.\\
& & \left.  j \sum_{\omega\in\hat{I}_N\setminus\{k,j\}} \omega | \hat q(\omega-k) \hat p(j-\omega) |\left(2\Delta t^2 \right)  + \right. \\
& & \left.   \sum_{\omega\in\hat{I}_N\setminus\{k,j\}} |\hat q(\omega-k) \hat q(j-\omega) |\left(2\Delta t^2 \right) \right|.
\end{eqnarray*}
From these bounds, we obtain
\begin{eqnarray}
\|G_{11}\|_\infty
&\leq&\max\limits_{1\leq j\leq N} \sum_{k\in\hat{I}_N\setminus j} (\bar{p}j^2+\bar{q})^{-1/2}\left|\bar{A}_{jk}+\bar{B}_{jk}+\bar{C}_{jk}+\bar{D}_{jk}\right| (\bar{p}k^2+\bar{q})^{-1/2}\nonumber \\
&\leq&\max\limits_{1\leq j\leq N} (\bar{p}j^2+\bar{q})^{-1/2}(\bar{p}j^2+\bar{q})^{-1/2} (\overline{p} j^2 +\overline{q}) + \nonumber \\
& & \frac{3}{2}\sum_{k\in\hat{I}_N\setminus j} \left| j k \hat p(k-j) (\bar{p}j^2+\bar{q})^{-1/2}(\bar{p}k^2+\bar{q})^{1/2}  \Delta t^2 \right| + \nonumber \\
& & \frac{3}{2}\sum_{k\in\hat{I}_N\setminus j} \left| \hat q(k-j) (\bar{p}j^2+\bar{q})^{-1/2}(\bar{p}k^2+\bar{q})^{1/2} \Delta t^2 \right|   +  \nonumber \\
& & \frac{3}{2}\sum_{k\in\hat{I}_N\setminus j} \left| j k \hat p(k-j) (\bar{p}j^2+\bar{q})^{1/2}(\bar{p}k^2+\bar{q})^{-1/2}  \Delta t^2 \right| + \nonumber \\
& & \frac{3}{2}\sum_{k\in\hat{I}_N\setminus j} \left| \hat q(k-j) (\bar{p}j^2+\bar{q})^{1/2}(\bar{p}k^2+\bar{q})^{-1/2} \Delta t^2 \right|   +  \nonumber \\
& & \sum_{k\in\hat{I}_N} \left| jk (\bar{p}j^2+\bar{q})^{-1/2}(\bar{p}k^2+\bar{q})^{-1/2} \sum_{\omega\in\hat{I}_N\setminus\{k,j\}} \omega^2 |\hat p(\omega-k)\hat p(j-\omega)|\Delta t^2   \right| + \nonumber \\
& &  \sum_{k\in\hat{I}_N} \left| k (\bar{p}j^2+\bar{q})^{-1/2}(\bar{p}k^2+\bar{q})^{-1/2} \sum_{\omega\in\hat{I}_N\setminus\{k,j\}} \omega | \hat p(\omega-k)  \hat q(j-\omega) |\Delta t^2 \right|   + \nonumber \\
& &  \sum_{k\in\hat{I}_N} \left| j (\bar{p}j^2+\bar{q})^{-1/2}(\bar{p}k^2+\bar{q})^{-1/2} \sum_{\omega\in\hat{I}_N\setminus\{k,j\}} \omega | \hat q(\omega-k) \hat p(j-\omega)|\Delta t^2 \right|  +\nonumber \\
& &  \sum_{k\in\hat{I}_N} \left| (\bar{p}j^2+\bar{q})^{-1/2} (\bar{p}k^2+\bar{q})^{-1/2}\sum_{\omega\in\hat{I}_N\setminus\{k,j\}} |\hat q(\omega-k) \hat q(j-\omega)|\Delta t^2  \right| \nonumber
\end{eqnarray}
\begin{eqnarray}
&\leq&\max\limits_{1\leq j\leq N} 1  +  \frac{3}{2} j \Delta t^2 \|\tilde{p}\|_{\infty} (\overline{p}j^2+\overline{q})^{-1/2} \sum_{k\in{I}_j} \left| k (\overline{p} k^2 +\overline{q})^{1/2} \right| +\nonumber \\
& & \frac{3}{2}\Delta t^2 \|\tilde{q}\|_{\infty} (\overline{p}j^2+\overline{q})^{-1/2} \sum_{k\in{I}_j} \left|(\overline{p} k^2 +\overline{q})^{1/2} \right|   +  \nonumber  \\
& & \frac{3}{2}j \Delta t^2 \|\tilde{p}\|_{\infty}   (\overline{p} j^2 +\overline{q})^{1/2} \sum_{k\in{I}_j} \left| k (\overline{p}k^2+\overline{q})^{-1/2}\right|  + \nonumber \\
& & \frac{3}{2}\Delta t^2 \|\tilde{q}\|_{\infty} (\overline{p}j^2+\overline{q})^{1/2} \sum_{k\in{I}_j} \left|(\overline{p}k^2+\overline{q})^{-1/2} \right|   + \nonumber   \\
& & j \Delta t^2 \|\tilde{p}\|_{\infty}^2 (\overline{p} j^2 +\overline{q})^{-1/2}  \sum_{k\in\hat{I}_N} \left| k (\overline{p} k^2 +\overline{q})^{-1/2} \sum_{\omega\in I_j\cap I_k} \omega^2 \right| + \nonumber  \\
& & \Delta t^2 \|\tilde{p}\|_{\infty} \|\tilde{q}\|_{\infty} (\overline{p} j^2 +\overline{q})^{-1/2} \sum_{k\in\hat{I}_N} \left| k (\overline{p} k^2 +\overline{q})^{-1/2} \sum_{\omega\in I_j\cap I_k} \omega \right| + \nonumber  \\
& &  j\Delta t^2 \|\tilde{p}\|_{\infty} \|\tilde{q}\|_{\infty} (\overline{p} j^2 +\overline{q})^{-1/2} \sum_{k\in\hat{I}_N} \left| (\overline{p} k^2 +\overline{q})^{-1/2}  \sum_{\omega\in I_j\cap I_k} \omega \right|  + \nonumber \\
& & \Delta t^2 \|\tilde{q}\|_{\infty}^2	(\overline{p} j^2 +\overline{q})^{-1/2}  \sum_{k\in\hat{I}_N} \left| (\overline{p} k^2 +\overline{q})^{-1/2}  \sum_{\omega\in I_j\cap I_k} 1 \right|.
\label{G11sums}
\end{eqnarray}

We can bound each of the summations in (\ref{G11sums}) as in the following examples.
\begin{itemize}
\item
We first derive a bound for
\begin{equation} \label{eqsumexp} (\overline{p}j^2+\overline{q})^{1/2}\sum_{k\in I_j}\left|  (\overline{p}k^2+\overline{q})^{-1/2} \right|.  
\end{equation}
For 
$|j|> \omega_{\max}$, we have 
\begin{eqnarray*}
(\overline{p} j^2 +\overline{q})^{1/2} \sum_{k\in I_j} \left| (\overline{p}k^2+\overline{q})^{-1/2}\right| &\leq& (\overline{p} j^2 +\overline{q})^{1/2} \sum_{k\in I_j} \frac{1}{(\overline{p}(|j|-\omega_{\max})^2+\overline{q})^{1/2}}  \\
&\leq& (\overline{p} j^2 +\overline{q})^{1/2} \frac{2\omega_{\max}}{(\overline{p}(|j|-\omega_{\max})^2+\overline{q})^{1/2}}. 
\end{eqnarray*}
As $|j|\rightarrow\infty$, we obtain 
$$\lim_{j\rightarrow\infty} (\overline{p}j^2+\overline{q})^{1/2} \sum_{k\in I_j}\left| (\overline{p}k^2+\overline{q})^{-1/2} \right| \leq 2 \omega_{\max}.$$
Therefore, the expression (\ref{eqsumexp}) can be bounded independently of $N$.
\item Next, we derive a bound for
\begin{equation} \label{eqsumexp2} (\overline{p} j^2 +\overline{q})^{-1/2} \sum_{k\in\hat{I}_N} \left| k (\overline{p} k^2 +\overline{q})^{-1/2} \sum_{\omega\in I_j\cap I_k} \omega \right|.
\end{equation}
We have
\begin{eqnarray*}
\sum_{k\in\hat{I}_N} \left| k (\overline{p} k^2 +\overline{q})^{-1/2} \sum_{\omega\in I_j\cap I_k} \omega \right| & \leq & \sum_{k\in\hat{I}_N} \sum_{\omega\in I_j\cap I_k} \left| k (\overline{p} k^2 +\overline{q})^{-1/2}  \omega \right| \\
& \leq & \sum_{\omega\in I_j} |\omega| \sum_{k\in I_\omega} \left| k (\overline{p} k^2 +\overline{q})^{-1/2}   \right| \\
& \leq & \sum_{\omega\in I_j} |\omega| \sum_{\eta\in I_0} \left| (\omega+\eta)(\overline{p} (\omega+\eta)^2 +\overline{q})^{-1/2}   \right|. \\
\end{eqnarray*}
If $j > 2\omega_{\max}$, then $\omega > \omega_{\max}$, and
\begin{eqnarray*}
\sum_{k\in\hat{I}_N} \left| k (\overline{p} k^2 +\overline{q})^{-1/2} \sum_{\omega\in I_j\cap I_k} \omega \right| & \leq & \sum_{\omega\in I_j} |j+\omega_{\max}| \sum_{\eta\in I_0} 
\left| \frac{j+2\omega_{\max}}{(\overline{p} (j-2\omega_{\max})^2 +\overline{q})^{1/2}}   \right|. \\
\end{eqnarray*}
We then have
$$\lim_{j\rightarrow\infty} (\overline{p} j^2 +\overline{q})^{-1/2} \sum_{k\in\hat{I}_N} \left| k (\overline{p} k^2 +\overline{q})^{-1/2} \sum_{\omega\in I_j\cap I_k} \omega \right| \leq \frac{1}{\bar{p}}.$$
We conclude that the expression (\ref{eqsumexp2}) can also be bounded independently of $N$.
\end{itemize}
Using a similar approach to bound the remaining summations in (\ref{G11sums}), we obtain
\begin{eqnarray*}
\|G_{11}\|_\infty &\leq&  1  +  C_{11,p} \|\tilde{p}\|_{\infty} \Delta t^2 N^2 +  C_{11,q} \|\tilde{q}\|_{\infty} \Delta t^2 + C_{11,p^2} \|\tilde{p}\|_{\infty}^2 \Delta t^2 N^2 +   \\
& & C_{11,pq} \|\tilde{p}\|_{\infty} \|\tilde{q}\|_{\infty} \Delta t^2 + C_{11,q^2} \|\tilde{q}\|_{\infty}^2 \Delta t^2,
\end{eqnarray*}
for constants $C_{11,p}$, $C_{11,q}$, $C_{11,p^2}$, $C_{11,pq}$, $C_{11,q^2}$ that are independent
of $N$ and $\Delta t$, which completes the proof.
\end{proof} 

\vspace{0.1in}
\noindent Using the same approach, we find that the matrix $G_{22}$ defined in (\ref{G}) satisfies a bound of the same form as that of $G_{11}$, with appropriate constant factors. 

\vspace{0.1in}
\begin{lem} \label{lem32}
Assume $\hat p(\omega)=0$ and $\hat q(\omega)=0$ for $|\omega|>\omega_{\max}$. Then the matrix $G_{12}$ defined in (\ref{G}) satisfies 
\begin{eqnarray}
\|G_{12}\|_{\infty} &\leq& C_{12,p} \|\tilde{p}\|_{\infty} \Delta t N + C_{12,q}  \|\tilde{q}\|_{\infty} \Delta t  + C_{12,p^2} \|\tilde{p}\|_{\infty}^2 \Delta t N + \nonumber \\ 
& & C_{12,pq} \|\tilde{p}\|_{\infty} \|\tilde{q}\|_{\infty} \Delta t + C_{12,q^2} \|\tilde{q}\|_{\infty}^2 \Delta t,
\label{G12bound}
\end{eqnarray}
where each constant is independent of $N$ and $\Delta t$.
\end{lem}
\begin{proof}
From (\ref{eqzintro}), we have
\begin{eqnarray*}
[{\bf u}^n]^T\bar{G}_{12}{\bf u}_t^n&=&\sum_{\omega\in\hat{I}_N} \overline{\hat{\bf z}_{11}(\omega)} \hat{\bf z}_{12}(\omega)(\bar{p}\omega^2+\bar{q}) + \sum_{\omega\in\hat{I}_N} \overline{\hat{\bf z}_{21}(\omega)} \hat{\bf z}_{22}(\omega) \\
& = & \sum_{j\in\hat{I}_N} \sum_{k\in\hat{I}_N} \hat u(-j) \hat u_t(k) \left[ \bar{A} + \bar{B} + \bar{C} + \bar{D} \right]_{jk},
\end{eqnarray*}
where, for $j,k\in\hat{I}_N$,
\begin{eqnarray*}
\bar{A}_{jj} &=& S_{11}(l_{2,j})S_{12}(l_{2,j}) (\bar{p}j^2+\bar{q}) + S_{21}(l_{2,j}) S_{22}(l_{2,j}) \\
&=&  \cos\left((\overline{p}j^2+\overline{q})^{1/2}\Delta t\right)  (\overline{p}j^2 +q)^{-1/2}\sin\left(\sqrt{\overline{p}j^2 +q}\Delta t\right)  (\bar{p}j^2+\bar{q}) +  \\
& &  \left(-(\overline{p}j^2+\overline{q})^{1/2}\sin\left((\overline{p}j^2+\overline{q})^{1/2}\Delta t\right) \cos\left(\sqrt{\overline{p}j^2+\overline{q}}\Delta t\right)\right)  \\
&=& \cos\left((\overline{p}j^2+\overline{q})^{1/2}\Delta t\right)  (\overline{p}j^2 +q)^{1/2}\sin\left(\sqrt{\overline{p}j^2 +q}\Delta t\right)  -  \\
& & (\overline{p}j^2+\overline{q})^{1/2}\sin\left((\overline{p}j^2+\overline{q})^{1/2}\Delta t\right) \cos\left(\sqrt{\overline{p}j^2+\overline{q}}\Delta t\right) \\
&=& 0,
\end{eqnarray*}
\begin{eqnarray*}
\bar{B}_{jk} &=&  -jk S_{12}(l_{2,k}) M_{11,k} \hat p(k-j) (\bar{p}k^2+\bar{q}) +   S_{12}(l_{2,k}) M_{11,k} \hat q(k-j)(\bar{p}k^2+\bar{q}) + \\
& &  -jk S_{22}(l_{2,k})M_{21,k} \hat p(k-j)  + S_{22}(l_{2,k})M_{21,k} \hat q(k-j) , \quad j\neq k,
\end{eqnarray*}
\begin{eqnarray*}
\bar{C}_{jk} &=&  -jk S_{11}(l_{2,j}) M_{12,j} \hat p(k-j) (\bar{p}j^2+\bar{q}) +  S_{11}(l_{2,j}) M_{12,j} \hat q(k-j) (\bar{p}j^2+\bar{q}) +  \\
& &  -jk S_{21}(l_{2,j})M_{22,j} \hat p(k-j) +   S_{21}(l_{2,j}) M_{22,j} \hat q(k-j), \quad j \neq k,
\end{eqnarray*}
and
\begin{eqnarray*}
\bar{D}_{jk} 
&=& -jk \sum_{\omega\in\hat{I}_N\setminus\{k,j\}} \omega^2  \hat p(j-\omega) \hat p(\omega-k) \left( M_{11,\omega} M_{12,\omega} (\bar{p}\omega^2+\bar{q}) +  M_{21,\omega} M_{22,\omega} \right) + \\
& &  (-j) \sum_{\omega\in\hat{I}_N\setminus\{k,j\}} \omega \hat p(j-\omega) \hat q(\omega-k) \left( M_{11,\omega} M_{12,\omega}(\bar{p}\omega^2+\bar{q}) +  M_{21,\omega} M_{22,\omega} \right) + \\
& &  (-k) \sum_{\omega\in\hat{I}_N\setminus\{k,j\}} \omega \hat q(j-\omega) \hat p(\omega-k) \left( M_{11,\omega} M_{12,\omega} (\bar{p}\omega^2+\bar{q}) +   M_{21,\omega} M_{22,\omega} \right) + \\
& &  \sum_{\omega\in\hat{I}_N\setminus\{k,j\}} \hat q(j-\omega) \hat q(\omega-k) \left( M_{11,\omega} M_{12,\omega} (\bar{p}\omega^2+\bar{q}) + M_{21,\omega} M_{22,\omega} \right),
\end{eqnarray*}
with $\bar{A}_{jk}=0$ for $j\neq k$, and $\bar{B}_{jj} = \bar{C}_{jj} = 0$ for $j\in\hat{I}_N$.

To obtain an upper bound for $\|G_{12}\|_{\infty}$, we use the following bounds on $S_{ij}(l_{2,\omega})$ and $M_{ij,\omega}:$
\begin{eqnarray*}
|S_{12}(l_{2,\omega})| &\leq& \left| l_{2,\omega}^{-1/2}\sin(l_{2,\omega}^{1/2}\Delta t)\right| \leq \left| l_{2,\omega}^{-1/2}\right| l_{2,\omega}^{1/2}\Delta t = \Delta t,\\ 
|M_{11,\omega}| &=& |M_{22,\omega}| \leq \frac{2}{\overline{p}\omega^2+\overline{q}}, \\
|M_{12,\omega}| &\leq& \dfrac{\Delta t}{\overline{p}\omega^2+\overline{q}}.
\end{eqnarray*}
Then we have
\begin{eqnarray*}
|\bar{B}_{jk}| 
&\leq&3 \left( \left|  j k \hat p(k-j) \Delta t \right| + \left|  \hat q(k-j) \Delta t   \right|\right),
\end{eqnarray*}
\begin{eqnarray*}
|\bar{C}_{jk}| 
&\leq& 3\left( \left|   j k \hat p(k-j) \Delta t \right| + \left|  \hat q(k-j) \Delta t  \right|\right),
\end{eqnarray*} 
and
\begin{eqnarray*}
|\bar{D}_{jk}| &\leq& \left| jk \sum_{\omega\in\hat{I}_N\setminus\{k,j\}} \omega^2 | \hat p(j-\omega) \hat p(\omega-k) |\frac{4\Delta t}{\overline{p}\omega^2+\overline{q}} + \right. \\
& & \left. j \sum_{\omega\in\hat{I}_N\setminus\{k,j\}} \omega |\hat p(j-\omega) \hat q(\omega-k) |\frac{4\Delta t}{\overline{p}\omega^2+\overline{q}} + \right. \\
& & \left.  k \sum_{\omega\in\hat{I}_N\setminus\{k,j\}} \omega |\hat q(j-\omega) \hat p(\omega-k) |\frac{4\Delta t}{\overline{p}\omega^2+\overline{q}} + \right. \\
& & \left.  \sum_{\omega\in\hat{I}_N\setminus\{k,j\}} |\hat q(j-\omega) \hat q(\omega-k) |\frac{4\Delta t}{\overline{p}\omega^2+\overline{q}}  \right|.
\end{eqnarray*}
From these bounds, we obtain
\begin{eqnarray*}
\|G_{12}\|_{\infty} &\leq& \max\limits_{1\leq j\leq N} \sum_{k\in\hat{I}_N\setminus j} \left|\left(\bar{A}_{jk}+\bar{B}_{jk}+\bar{C}_{jk}+\bar{D}_{jk}\right)(\bar{p}j^2+\bar{q})^{-1/2}\right|\\
&\leq&  \max\limits_{1\leq j\leq N} 6\sum_{k\in\hat{I}_N\setminus j} \left|  j k \hat p(k-j) (\bar{p}j^2+\bar{q})^{-1/2} \Delta t   \right| + \\
& & 6\sum_{k\in\hat{I}_N\setminus j} \left|  \hat q(k-j) (\bar{p}j^2+\bar{q})^{-1/2} \Delta t \right| + \\
& & \sum_{k\in\hat{I}_N\setminus j} \left|   jk \sum_{\omega\in\hat{I}_N\setminus\{k,j\}}  \omega^2 |\hat p(j-\omega) \hat p(\omega-k) |(\bar{p}j^2+\bar{q})^{-1/2} \frac{4\Delta t}{\overline{p}\omega^2+\overline{q}} \right| +  \\
& & \sum_{k\in\hat{I}_N\setminus j} \left|  j \sum_{\omega\in\hat{I}_N\setminus\{k,j\}} \omega |\hat p(j-\omega) \hat q(\omega-k) |(\bar{p}j^2+\bar{q})^{-1/2} \frac{4\Delta t}{\overline{p}\omega^2+\overline{q}} \right|  + \\
& & \sum_{k\in\hat{I}_N\setminus j} \left|  k\sum_{\omega\in\hat{I}_N\setminus\{k,j\}} \omega |\hat q(j-\omega) \hat p(\omega-k) |(\bar{p}j^2+\bar{q})^{-1/2} \frac{4\Delta t}{\overline{p}\omega^2+\overline{q}} \right|  + \\
& & \sum_{k\in\hat{I}_N\setminus j} \left|  \sum_{\omega\in\hat{I}_N\setminus\{k,j\}} |\hat q(j-\omega) \hat q(\omega-k) |(\bar{p}j^2+\bar{q})^{-1/2} \frac{4\Delta t}{\overline{p}\omega^2+\overline{q}} \right|  \\
&\leq&  \max\limits_{1\leq j\leq N}  6j \|\tilde{p}\|_{\infty} (\bar{p}j^2+\bar{q})^{-1/2} \Delta t  \sum_{k\in I_j} \left| k \right|   +   \\
& & 6\|\tilde{q}\|_{\infty} (\bar{p}j^2+\bar{q})^{-1/2} \Delta t  (2\omega_{max}) + \\
& &  j \|\tilde{p}\|_{\infty}^2 (\bar{p}j^2+\bar{q})^{-1/2} \Delta t \frac{4}{\overline{p}} \sum_{k\in\hat{I}_N} \sum_{\omega\in I_j \cap I_k} |k| +  \\
& & \|\tilde{p}\|_{\infty} \|\tilde{q}\|_{\infty} (\bar{p}j^2+\bar{q})^{-1/2} \Delta t  \frac{4}{\overline{p}}   \sum_{\omega\in I_j}\left( 2j\omega_{max}  + \sum_{k\in I_\omega} \left|k \right| \right)+ \\
& & \|\tilde{q}\|_{\infty}^2 (\bar{p}j^2+\bar{q})^{-1/2} \Delta t \frac{4}{\overline{q}} (4\omega_{max}^2).
\end{eqnarray*}

These summations can be bounded as in the proof of Lemma \ref{lem31}.  As an example,
we obtain a bound for
\begin{equation} \label{eqsumexp3}
(\bar{p}j^2+\bar{q})^{-1/2} \sum_{k\in I_j}\left| k \right|.
\end{equation}
If $|j| > \omega_{\max}$, then 
\begin{equation}
\sum_{k\in I_j}\left| k \right|=\sum_{\eta\in I_0} \left| j-\eta \right|=\sum_{\eta=1}^{\omega_{\max}} \left|j-\eta\right| + |j+\eta|=\sum_{\eta=1}^{\omega_{\max}}2|j|=2|j|\omega_{\max}.\label{sumk2} \end{equation}
It follows that
$$\lim_{j\rightarrow\infty} (\bar{p}j^2+\bar{q})^{-1/2} \sum_{k\in I_j}\left| k \right| 
=\lim_{j\rightarrow\infty} (\bar{p}j^2+\bar{q})^{-1/2}2|j|\omega_{\max} = \frac{2\omega_{\max}}{\bar{p}}.$$
That is, the expression (\ref{eqsumexp3}) is bounded independently of $N$, whereas it would be
$O(N)$ if $p(x)$ was not bandlimited.

Proceeding in a similar manner for the remaining summations, we conclude that there exist
constants $C_{12,p}$, $C_{12,q}$, $C_{12,p^2}$, $C_{12,pq}$ and $C_{12,q^2}$ such that
\begin{eqnarray*}
\|G_{12}\|_{\infty} &\leq& C_{12,p} \|\tilde{p}\|_{\infty} \Delta t N + C_{12,q}  \|\tilde{q}\|_{\infty} \Delta t  + C_{12,p^2} \|\tilde{p}\|_{\infty}^2 \Delta t N + \\
& & C_{12,pq} \|\tilde{p}\|_{\infty} \|\tilde{q}\|_{\infty} \Delta t + C_{12,q^2} \|\tilde{q}\|_{\infty}^2 \Delta t,
\end{eqnarray*}
which completes the proof.
\end{proof}

\vspace{0.1in}
\noindent Using the same approach, it can be shown that 
the matrix $G_{21}$ in (\ref{G}) satisfies a bound of the same form
as that of $G_{12}$.

We now prove a result that gives us reason to believe that 
the second-order KSS method applied to (\ref{pde}), (\ref{ICs}), (\ref{BCs}) may be unstable. 

\vspace{0.1in}
\begin{thm} \label{thm33}
Assume $\hat p(\omega)=0$ and $\hat q(\omega)=0$ for $|\omega|>\omega_{\max}$. Then the solution operator $S_N(\Delta t)$ satisfies 
\begin{eqnarray} \label{eqthm33}
\|S_N(\Delta t)\|_{C_N}\leq1+(C_p\|\tilde{p}\|_\infty N+C_q\|\tilde{q}\|_\infty )\Delta t,
\end{eqnarray}
where the constants $C_p$ and $C_q$ are independent of $N$ and $\Delta t$.
\end{thm}
\begin{proof} From
\begin{equation}
\left\| \begin{bmatrix}
G_{11} & G_{12} \\
G_{21} & G_{22}
\end{bmatrix} \right\|_{\infty}
\leq \max\{\|G_{11}\|_{\infty} + \|G_{12}\|_{\infty}, \|G_{21}\|_{\infty} + \|G_{22}\|_{\infty}\},
\end{equation}
we have, for some $k\in\{1,2\}$,
$$\|S_N(\Delta t)\|_{C_N}=\|B\|_2 \leq \sqrt{\|G\|_{\infty}} \leq \sqrt{\|G_{k1}\|_{\infty} + \|G_{k2}\|_{\infty}}.$$
From Lemma \ref{lem31} and Lemma \ref{lem32}, we have 
\begin{eqnarray}
\|G\|_{\infty} &\leq& 1 + \Delta t N \left( C_{k2,p} \|\tilde{p}\|_{\infty}+C_{k2,p^2} \|\tilde{p}\|_{\infty}^2 \right) +  \Delta t^2 N^2 \left( C_{k1,p} \|\tilde{p}\|_{\infty}+C_{k1,p^2} \|\tilde{p}\|_{\infty}^2 \right) + \nonumber \\ 
& & \Delta t \left( C_{k2,q}  \|\tilde{q}\|_{\infty}  + C_{k2,pq} \|\tilde{p}\|_{\infty} \|\tilde{q}\|_{\infty} + C_{k2,q^2} \|\tilde{q}\|_{\infty}^2 \right) + \nonumber \\ 
& & \Delta t^2 \left( C_{k1,q} \|\tilde{q}\|_{\infty} + C_{k1,pq} \|\tilde{p}\|_{\infty} \|\tilde{q}\|_{\infty} + C_{k1,q^2} \|\tilde{q}\|_{\infty}^2 \right).
\label{normofG}
\end{eqnarray}
Let $$R_1=N \left( C_{k2,p} \|\tilde{p}\|_{\infty}+C_{k2,p^2} \|\tilde{p}\|_{\infty}^2 \right)+C_{k2,q}  \|\tilde{q}\|_{\infty}  + C_{k2,pq} \|\tilde{p}\|_{\infty} \|\tilde{q}\|_{\infty} + C_{k2,q^2} \|\tilde{q}\|_{\infty}^2,$$ 
$$R_2=N^2 \left( C_{k1,p} \|\tilde{p}\|_{\infty}+C_{k1,p^2} \|\tilde{p}\|_{\infty}^2 \right)+C_{k1,q} \|\tilde{q}\|_{\infty} + C_{k1,pq} \|\tilde{p}\|_{\infty} \|\tilde{q}\|_{\infty} + C_{k1,q^2} \|\tilde{q}\|_{\infty}^2,$$
and $R=\max\{R_2^{1/2},R_1/2\}$. 
We then have
$$\|S_N(\Delta t)\|_{C_N}\leq\sqrt{\|G\|_\infty}\leq 1+R\Delta t\leq1+(C_p\|\tilde{p}\|_\infty N+C_q\|\tilde{q}\|_\infty )\Delta t,$$
from which the result follows. 
\end{proof}

\vspace{0.1in}
While Theorem \ref{thm33} does not prove that the bound in (\ref{eqthm33}) is sharp,
numerical experiments indicate that it actually is.  In the case of $p(x)\equiv \textrm{constant}$,
we obtain a more favorable stability result.

\vspace{0.1in}
\begin{cor} \label{corstab}
\noindent Assume the leading coefficient $p(x)$ is constant. Then, under the assumptions of Theorem
\ref{thm33},
$$\|S_N(\Delta t)\|_{C_N}\leq1+C_q\|\tilde{q}\|_\infty\Delta t.$$
\end{cor}
\begin{proof}
Because the leading coefficient $p(x)$ is constant, we have $\tilde{p}(x)=p(x)-\bar{p}\equiv 0$. 
The result follows immediately from the last line of the proof of Theorem \ref{thm33}.
\end{proof}

\vspace{0.1in}
\noindent Therefore, a second-order KSS method applied to (\ref{pde}), (\ref{ICs}), (\ref{BCs}), (\ref{eqLform}), under the assumptions that $p(x)$ is constant and $q(x)$ is bandlimited, is unconditionally stable.

\subsection{Consistency}

For the remainder of this convergence analysis, we assume the coefficient $p(x)$
from (\ref{eqLform}) is constant,
since stability has been proved only for this case.

Before we obtain an estimate of the local truncation error, we introduce additional notation.  
We first define the restriction operator
$${\cal R}_N {\bf f}(x) = {\bf f}({\bf x}_N),$$
and interpolation operator
$${\cal T}_N {\bf g} = {\cal T}_N \left[ \begin{array}{c} {\bf g}_1 \\ {\bf g}_2 \end{array} \right]
= \left[ \begin{array}{c} 
\displaystyle{\sum_{\omega=-N/2+1}^{N/2} e^{i\omega x} \tilde{g}_1(\omega)} \\
\\
\displaystyle{\sum_{\omega=-N/2+1}^{N/2} e^{i\omega x} \tilde{g}_2(\omega)} \\
\end{array} \right],$$
where, for $i=1,2,$
$$\tilde{g}_i(\omega) = \frac{1}{N} \sum_{j=1}^{N} e^{-i\omega x_j} [{\bf g}_i]_j.$$
Then, the operator ${\cal I}_N = {\cal T}_N {\cal R}_N$ on $L^2([0,2\pi]) \times L^2([0,2\pi])$
computes the Fourier interpolant of each component function.  By contrast, if we define
$$\hat{\cal R}_N {\bf f}(x) = \hat{\cal R}_N \left[ \begin{array}{c} f_1(x) \\ f_2(x) \end{array} \right] =
\left[ \begin{array}{c}
\displaystyle{\sum_{\omega=-N/2+1}^{N/2} e^{i\omega {\bf x}_N} \hat{f}_1(\omega)} \\
\\
\displaystyle{\sum_{\omega=-N/2+1}^{N/2} e^{i\omega {\bf x}_N} \hat{f}_2(\omega)} \\
\end{array} \right],$$
where, for $i=1,2,$
$$\hat{f}_i(\omega) = \frac{1}{2\pi} \int_0^{2\pi} e^{-i\omega x} f_i(x)\,dx,$$
then the operator ${\cal P}_N = {\cal T}_N \hat{\cal R}_N$ on $L^2([0,2\pi]) \times L^2([0,2\pi])$
is the orthogonal projection operator onto $\textrm{span}
\{ e^{i\omega x} \}_{\omega=-N/2+1}^{N/2}$.
Finally, the continuous approximate solution operator $\tilde{S}_N(\Delta t) : 
L^2([0,2\pi]) \rightarrow L^2([0,2\pi])$ is defined by $\tilde{S}_N(\Delta t) = {\cal T}_N {S}_N(\Delta t) {\cal R}_N$.

\begin{thm} \label{thmconsist}
Let ${\bf f}\in H_p^{m+1}([0,2\pi]) \times H_p^m([0,2\pi])$ for $m \geq 4$.  Then, for the problem (\ref{pde}), (\ref{ICs}), (\ref{BCs}), (\ref{eqLform}) on the domain
$(0,2\pi)\times (0,T)$, with $p(x)$ constant and $q(x)$ bandlimited, the two-node block KSS method with prescribed nodes (\ref{eqinterpw}) is consistent.  That is, 
the local truncation error satisfies
$$\frac{1}{\Delta t} \left\|{\cal I}_N \exp[\tilde{L}\Delta t]{\bf f} - \tilde{S}_N(\Delta t){\bf f} \right\|_{C} \leq C_1\Delta x^{m-1} + C_2\Delta t^2,$$
where 
the constants $C_1$ and $C_2$ are independent of 
$\Delta x$ and $\Delta t$.
\end{thm}

{\em Proof:}
We split the local truncation error into two components:
\begin{eqnarray*}
E_1(\Delta t,\Delta x) & = & {\cal I}_N \exp(\tilde{L}\Delta t){\bf f}(x) -  {\cal T}_N
\exp(\tilde{L}_N \Delta t){\cal R}_N{\bf f}(x) \\
E_2(\Delta t,\Delta x) & = & {\cal T}_N \exp(\tilde{L}_N \Delta t){\cal R}_N{\bf f}(x) -
 \tilde{S}_N(\Delta t){\bf f}(x) \\
& = & {\cal T}_N [\exp(\tilde{L}_N\Delta t) - S_N(\Delta t)]{\cal R}_N {\bf f}(x).
\end{eqnarray*}
First, we bound $E_1(\Delta t,\Delta x)$.  Because of the regularity of ${\bf f}$, we have
\begin{equation} \label{eqfreg}
\|{\bf f} - {\bf f}_N\|_C \leq C_0\Delta x^m
\end{equation}
for some constant $C_0$ (see \cite[Theorem 2.16]{spectime}).  Next, we note that
the exact solution ${\bf v}(x,t) = \exp[\tilde{L} t]{\bf f}(x)$ has the spectral decomposition
$${\bf v}(x,t) = \sum_{k=1}^\infty e^{\mu_k t} {\bf v}_k(x) \langle {\bf v}_k, {\bf f} 
\rangle$$
where $\{ \mu_k \}_{k=1}^\infty$ are the purely imaginary eigenvalues of $\tilde{L}$, and $\{ {\bf v}_k(x) \}_{k=1}^\infty$
are the corresponding orthonormal eigenfunctions, each of which belongs to $C_p^\infty[0,2\pi]$.  Using
this spectral decomposition, it can be shown using an approach similar to that used in 
\cite[Section 7.2, Theorem 6]{evans} for other hyperbolic PDEs that if ${\bf f}\in H_p^{m+1}([0,2\pi]) \times H_p^m([0,2\pi])$, then
${\bf v}(x,t)\in L^2(0,T,H_p^{m+1}([0,2\pi])\times H_p^{m}([0,2\pi]))$.  That is, the regularity of 
${\bf f}(x)$ is preserved in ${\bf v}(x,\Delta t)$ for each fixed $\Delta t > 0$. 
Therefore, there exists a constant $C_T$ such that
\begin{equation} \label{eqsolreg}
\|(I-{\cal I}_N){\bf v}(\cdot,\Delta t)\|_{C} \leq C_T \Delta x^{m}, \quad 0 < \Delta t \leq T.
\end{equation}

Using an approach based on \cite{bardostadmor} and applied in \cite{paper2}, we write 
$E_1(\Delta t,\Delta x)$ as
\begin{eqnarray*}
{\bf e}_N(x,t) &=& {\cal I}_N {\bf v}(x,t) - {\cal T}_N
\exp(\tilde{L}_N t){\cal R}_N{\bf f}(x) \\
& = & {\cal I}_N {\bf v}(x,t) - \exp({\cal I}_N \tilde{L} t){\cal I}_N{\bf f}(x)
\end{eqnarray*}
Then, ${\bf e}_N(x,t)$ solves the IVP
\begin{eqnarray*}
\frac{\partial}{\partial t} {\bf e}_N 
&=&  {\cal I}_N \tilde{L} {\bf e}_N + {\cal I}_N \tilde{L}(I - {\cal I}_N ) {\bf v}, \quad {\bf e}_N(x,0) = {\bf 0},
\end{eqnarray*}
and therefore
$${\bf e}_N(x,\Delta t) = \int_0^{\Delta t} e^{{\cal I}_N \tilde{L}(\Delta t- \tau)} {\cal I}_N  \tilde{L}(I - {\cal I}_N){\bf v}(x,\tau)\,d\tau.$$
From (\ref{eqsolreg}), and applying \cite[Section 7.2, Theorem 6]{evans}, it follows that
\begin{eqnarray*}
 \|E_1(\Delta t,\Delta x)\|_{C} & = & \|{\bf e}_N(x,\Delta t)\|_{C} \\
& \leq & \Delta t \max_{0\leq\tau\leq\Delta t} \left\|{\cal I}_N \tilde{L}
e^{{\cal I}_N \tilde{L}(\Delta t- \tau)} (I - {\cal I}_N){\bf v}(\cdot,\tau)\right\|_{C} \\
& \leq & C_1\Delta t\Delta x^{m-1},
\end{eqnarray*}
where the constant $C_1$ is independent of $\Delta x$ and $\Delta t$.  Here we note that
because the coefficients of $L$ are assumed to be constant or bandlimited, $E_1(\Delta t, \Delta x)$ does not
include aliasing error.


Now, we examine $E_2(\Delta t,\Delta x)$.  If we let
$$
{\cal R}_N {\bf f} = \left[ \begin{array}{c} {\bf f}_{N,1} \\ {\bf f}_{N,2} \end{array} \right], \quad
E_2(\Delta t,\Delta x) = {\cal T}_N \left[
\begin{array}{c} {\bf e}_{N,1} \\ {\bf e}_{N,2}
\end{array}
\right],
$$
then we have
\begin{eqnarray*}
{\bf e}_{N,1} & = & [\cos(L_N^{1/2} \Delta t) - S_{N,11}(\Delta t)] {\bf f}_{N,1} + 
 [L_N^{-1/2} \sin(L_N^{1/2} \Delta t) - S_{N,12}(\Delta t)] {\bf f}_{N,2}, \\
{\bf e}_{N,2} & = & [-L_N^{1/2} \sin(L_N^{1/2} \Delta t) - S_{N,21}(\Delta t)] {\bf f}_{N,1} + 
 [\cos(L_N^{1/2} \Delta t) - S_{N,22}(\Delta t)  ]{\bf f}_{N,2}.
\end{eqnarray*}
We have, by Parseval's identity,
$$\|E_2(\Delta t,\Delta x)\|_{C}^2 = 2\pi \sum_{\omega=-N/2+1}^{N/2} (\overline{p}\omega^2+\overline{q}) \left|  \frac{1}{N}
\hat{\bf e}_\omega^H {\bf e}_{N,1} \right|^2 + \left|  \frac{1}{N}
\hat{\bf e}_\omega^H {\bf e}_{N,2} \right|^2.$$
For each $\omega=-N/2+1,\ldots,N/2$, and $i,j=1,2$, we use the polynomial
interpolation error in $S_{N,ij}$ to obtain
\begin{eqnarray*}
\hat{\bf e}_\omega^H {\bf e}_{N,1}
& = & \frac{1}{2} \left. \frac{\partial^2}{\partial\lambda^2} \left[ \cos(\lambda^{1/2} \Delta t) \right]\right|_{\lambda=\xi_{\omega,11}} \hat{\bf e}_\omega^H (L_N - l_{1,\omega} I) (L_N - l_{2,\omega} I) {\bf f}_{N,1} + \\
& & \frac{1}{2} \left. \frac{\partial^2}{\partial\lambda^2} \left[ \lambda^{-1/2} \sin(\lambda^{1/2} \Delta t) \right]\right|_{\lambda=\xi_{\omega,12}} \hat{\bf e}_\omega^H (L_N - l_{1,\omega} I) (L_N - l_{2,\omega} I) {\bf f}_{N,2}, \\
\hat{\bf e}_\omega^H {\bf e}_{N,2}, 
& = & \frac{1}{2} \left. \frac{\partial^2}{\partial\lambda^2} \left[-\lambda^{1/2} \sin(\lambda^{1/2} \Delta t) \right]\right|_{\lambda=\xi_{\omega,21}} \hat{\bf e}_\omega^H (L_N - l_{1,\omega} I) (L_N - l_{2,\omega} I) {\bf f}_{N,1} + \\
& & \frac{1}{2} \left. \frac{\partial^2}{\partial\lambda^2} \left[  \cos(\lambda^{1/2} \Delta t) \right]\right|_{\lambda=\xi_{\omega,22}} \hat{\bf e}_\omega^H (L_N - l_{1,\omega} I) (L_N - l_{2,\omega} I) {\bf f}_{N,2},
\end{eqnarray*}
where $\xi_{\omega,ij} \in[l_{1,\omega},l_{2,\omega}]$ for $i,j=1,2$.
In view of the regularity of ${\bf f}(x)$, and the fact that $L_N$ is a discretization
of a second-order differential operator with bandlimited coefficients, and a constant leading coefficient,
we have
$$
\left| \frac{1}{N}\hat{\bf e}_0^H (L_N - l_{1,0} I) (L_N - l_{2,0} I) {\bf f}_{N,j}\right|
= \left| \frac{1}{N} (L_N \tilde{\bf q}_N)^T {\bf f}_{N,j}\right| \leq \|L_N \tilde{\bf q}_N\|_\infty
\|{\bf f}_{N,j}\|_\infty.
$$
It follows that there exist constants $C_{ij}$, for $i,j=1,2$, independent of $N$ and $\Delta t$, such that 
\begin{eqnarray*}
\bar{q}^{1/2} \left| \frac{1}{N} \hat{\bf e}_0^H {\bf e}_{N,1} \right| &\leq& \Delta t^4  C_{11}
+  \Delta t^5  C_{12}, \\
\left| \frac{1}{N}\hat{\bf e}_0^H {\bf e}_{N,2} \right| &\leq& \Delta t^3  C_{21}
+  \Delta t^4  C_{22},
\end{eqnarray*}
and for $\omega=-N/2+1,\ldots,-1,1,\ldots,N/2$, by Taylor expansion of the sines and cosines in $\hat{\bf e}_\omega^H {\bf e}_{N,1}$
and $\hat{\bf e}_\omega^H {\bf e}_{N,2}$, we have
\begin{eqnarray*}
(\bar{p}\omega^2+\bar{q})^{1/2} \left| \frac{1}{N}\hat{\bf e}_\omega^H {\bf e}_{N,1} \right| &\leq& \Delta t^4  \frac{C_{11}}{|\omega|^{m-2}}
+  \Delta t^5  \frac{C_{12}}{|\omega|^{m-3}}, \\
\left| \frac{1}{N}\hat{\bf e}_\omega^H {\bf e}_{N,2} \right| &\leq& \Delta t^3  \frac{C_{21}}{|\omega|^{m-1}}
+  \Delta t^4  \frac{C_{22}}{|\omega|^{m-2}}.
\end{eqnarray*}
It is important to note that because the leading coefficient $p(x)$ of $L$ is constant,
$(L_N - l_{2,\omega}I)\hat{\bf e}_\omega = \tilde{\bf q}_N\circ\hat{\bf e}_\omega$,
where $\circ$ denotes componentwise multiplication.  Therefore, this expression is bounded
independently of $\omega$.

Finally, we obtain
\begin{eqnarray*}
\|E_2(\Delta t,\Delta x)\|_{C} & \leq & \frac{\Delta t^3}{\sqrt{2\pi}} \left[ 
C_{11}^2 \Delta t^2 \left( 1 + 2\sum_{\omega=1}^\infty \omega^{4-2m} \right) +
2C_{11}C_{12} \Delta t^3 \left( 1 + 2\sum_{\omega=1}^\infty \omega^{5-2m} \right) + \right.\\&&\left.
C_{12}^2 \Delta t^4 \left( 1 + 2\sum_{\omega=1}^\infty \omega^{6-2m} \right) +
C_{21}^2 \left( 1 + 2\sum_{\omega=1}^\infty \omega^{2-2m} \right) + \right.\\&&\left.
2C_{21}C_{22} \Delta t \left( 1 + 2\sum_{\omega=1}^\infty \omega^{3-2m} \right) +
C_{22}^2 \Delta t^2 \left( 1 + 2\sum_{\omega=1}^\infty \omega^{4-2m} \right) 
\right]^{1/2} \\
& \leq & C_2\Delta t^3.
\end{eqnarray*}
Since $m \geq 4$, it follows that all of the summations over $\omega$ converge
to a sum that can be bounded independently of $N$.
We conclude that the constant $C_2$ is independent of $\Delta x$ and $\Delta t$.
$\Box$

\subsection{Convergence}

Now we can prove that the second-order KSS method converges
for the problem (\ref{pde}), (\ref{ICs}), (\ref{BCs}), (\ref{eqLform}) in the case 
of $p(x)$ being constant and $q(x)$ bandlimited.
We say that a method is convergent of order $(m,n)$ if there exist constants
$C_t$ and $C_x$, independent of the time step $\Delta t$ and grid spacing $\Delta x= 2\pi/N$, such that
$$\|{\bf u}(\cdot,t) - {\bf u}_N(\cdot,t)\|_{C} \leq C_t \Delta t^m + C_x \Delta x^n, \quad 0 \leq t \leq T,$$
where ${\bf u}(x,t)$ is the exact solution and ${\bf u}_N(x,t)$ is the approximate solution computed using
an $N$-point grid.

\begin{thm} \label{thmconv}
Under the assumptions of Theorem \ref{thmconsist},
the two-node block KSS method with prescribed nodes (\ref{eqinterpw}) is convergent of order $(2,m-1)$.
\end{thm}

{\em Proof:} We recall that $S(\Delta t)$ from (\ref{P}) is the exact solution operator for the problem (\ref{pde}), (\ref{ICs}), (\ref{BCs}), (\ref{eqLform}).
For any nonnegative integer $n$ and fixed grid size $N$, we define
$$E_n = \|{\cal I}_N S(\Delta t)^n {\bf f} - \tilde{S}_N(\Delta t)^n {\cal I}_N {\bf f}\|_{C}.$$
Then, by Theorem \ref{thmconsist} and Corollary \ref{corstab}, there exist constants
$C_1$, $C_2$ and $C_q$ such that
\begin{eqnarray*}
E_{n+1} & = & \|{\cal I}_N S(\Delta t)^{n+1} {\bf f} - \tilde{S}_N(\Delta t)^{n+1} {\cal I}_N {\bf f}\|_{C} \\
& = & \|{\cal I}_N S(\Delta t)S(\Delta t)^n {\bf f} - 
\tilde{S}_N(\Delta t)\tilde{S}_N(\Delta t)^n {\cal I}_N {\bf f}\|_{C} \\
& = &  \|{\cal I}_N S(\Delta t)S(\Delta t)^n {\bf f} - \tilde{S}_N(\Delta t)S(\Delta t)^n {\bf f} + 
\tilde{S}_N(\Delta t)S(\Delta t)^n {\bf f} - \tilde{S}_N(\Delta t)\tilde{S}_N(\Delta t)^n  {\cal I}_N {\bf f}\|_{C} \\
& \leq &  \|{\cal I}_N S(\Delta t)S(\Delta t)^n {\bf f} - \tilde{S}_N(\Delta t)S(\Delta t)^n {\bf f}\|_{C} + 
\|\tilde{S}_N(\Delta t)[{\cal I}_N S(\Delta t)^n {\bf f} - \tilde{S}_N(\Delta t)^n {\cal I}_N {\bf f}]\|_{C} \\
& \leq &  \|{\cal I}_N S(\Delta t)S(\Delta t)^n {\bf f} - \tilde{S}_N(\Delta t)S(\Delta t)^n {\bf f}\|_{C} + 
\|{\cal T}_N S_N(\Delta t){\cal R}_N [{\cal I}_N S(\Delta t)^n {\bf f} - 
\tilde{S}_N(\Delta t)^n {\cal I}_N {\bf f}]\|_{C} \\
& \leq &  \|{\cal I}_N S(\Delta t)S(\Delta t)^n {\bf f} - \tilde{S}_N(\Delta t)S(\Delta t)^n {\bf f}\|_{C} + 
\|{\cal T}_N S_N(\Delta t)\hat{\cal R}_N [{\cal I}_N S(\Delta t)^n {\bf f} - 
\tilde{S}_N(\Delta t)^n {\cal I}_N {\bf f}]\|_{C} \\
& \leq &  \|{\cal I}_N S(\Delta t){\bf u}(\cdot,t_n) - \tilde{S}_N(\Delta t){\cal I}_N {\bf u}(\cdot,t_n)\|_{C} + 
\|{S}_N(\Delta t)\|_{C_N} E_n \\
& \leq & C_1\Delta t^3 + C_2 \Delta t \Delta x^{m-1} + (1+C_q\|\tilde{q}\|_\infty\Delta t) E_n.
\end{eqnarray*}
It follows that
$$E_n \leq \frac{e^{C_q\|\tilde{q}\|_\infty T} - 1}{1+{C_q\|\tilde{q}\|_\infty \Delta t} - 1}(C_1\Delta t^3 + C_2 \Delta t\Delta x^{m-1})
\leq \tilde{C}_1\Delta t^2 + \tilde{C}_2\Delta x^{m-1}$$
for some constants $\tilde{C}_1$ and $\tilde{C}_2$ that depend only on $T$.
We conclude that
\begin{eqnarray*}
\|{\bf u}(\cdot,t_n) - {\bf u}_N(\cdot,t_n)\|_{C} & \leq & 
\|{\cal I}_N {\bf u}(\cdot,t_n) - {\bf u}_N(\cdot,t_n)\|_{C} + \|(I - {\cal I}_N){\bf u}(\cdot,t_n)\|_{C} \\
& \leq & \tilde{C}_1\Delta t^2 + \tilde{C}_2\Delta x^{m-1} + \tilde{C}_3 \Delta x^m.
\end{eqnarray*}
$\Box$

\section{Numerical Experiments}

We now perform some numerical experiments to corroborate the theory presented
in Section \ref{secconv}.
For each test case, relative error was estimated using the $\ell_2$ norm, in comparison
to a reference solution computed by the {\sc Matlab} ODE solver {\tt ode15s} \cite{shampine}, with
absolute and relative tolerances set to $10^{-12}$.

\subsection{Constant Leading Coefficient}

We consider the initial value problem
\begin{equation} \label{eqnumexpPDE}
u_{tt} + Lu = 0, \quad 0 < x < 2\pi, \quad t > 0,
\end{equation}
where $L$ is of the form (\ref{eqLform}), with
\begin{eqnarray}
p(x) &=& 1, \label{eqnumexpP} \\
q(x)&=&1+\frac{1}{2}\sin x+\frac{1}{4}\cos 2x +\frac{1}{8}\sin 3x. \label{eqnumexpQ}
\end{eqnarray}
The initial conditions are
\begin{eqnarray}
u(x,0) &=& \left\{
\begin{array}{ll}
1-\frac{2}{\pi}|x-\pi| & \pi/2 \leq x \leq 3\pi/2,\\
0 &  0 \leq x < \pi/2. \quad 3\pi/2 < x < 2\pi,
\end{array}
\right. \label{eqnumexpF} \\
u_t(x,0) &=& 0, \quad 0 < x < 2\pi, \label{eqnumexpG}
\end{eqnarray}
and we impose periodic boundary conditions.  We note that the initial data
belongs to $H_p^{m+1}([0,2\pi]) \times H_p^m([0,2\pi])$ for $m=1$, which is not
sufficiently regular to satisfy the assumptions of Theorem \ref{thmconsist}.

The results are shown in Table \ref{tabT2}.  As predicted by Theorem \ref{thmconsist}
and Corollary \ref{corstab}, we observe second-order accuracy in time, in spite of the low
regularity of the initial data,
even when the CFL limit is exceeded by using the same time step as the spatial
resolution increases.
\begin{table}[ht]
\caption{Relative errors in the solution of
(\ref{eqnumexpPDE}), (\ref{eqnumexpP}), (\ref{eqnumexpQ}), (\ref{eqnumexpF}), (\ref{eqnumexpG})
with periodic boundary conditions
on the domain $(0,2\pi)\times(0,10)$,
using the second-order KSS method described in Section \ref{secconv}, with $N$ grid points
and time step $\Delta t$.}
\begin{center}
\footnotesize
\begin{tabular}{|l|r|r|r|r|} \hline
$\Delta t$ & $N=256$ & $N=512$ & $N=1024$ & $N=2048$ \\ \hline
$\pi/128$ & 1.38e-04 & 1.33e-04 & 1.30e-04 & 1.29e-04 \\ 
$\pi/256$ & 3.32e-05 & 3.24e-05 & 3.25e-05  & 3.20e-05\\
$\pi/512$ & 8.04e-06 & 8.49e-06 & 8.27e-06  & 8.08e-06 \\ \hline
\end{tabular}
\normalsize
\end{center}
\label{tabT2}
\end{table}

\subsection{Variable Leading Coefficient}

We now solve the problem (\ref{eqnumexpPDE}), with 
\begin{equation}
p(x) = 1-\frac{1}{2}\sin x+\frac{1}{4}\cos 2x \label{eqnumexpP2}
\end{equation}
and initial conditions
\begin{equation}
u(x,0) = e^{-(x-\pi)^2}, \quad u_t(x,0) = 0, \quad 0 < x < 2\pi. \label{eqnumexpFG}
\end{equation}
The results are shown in Figures \ref{figinstability} and \ref{figstability}.
We see that when the CFL number is greater than one, the method is unstable,
as high-frequency components quickly become the dominant terms of the solution,
and their amplitudes grow without bound.  On the other hand, when the CFL number
is less than one, the solution is well-behaved.

\begin{figure}[ht]
\begin{center}
\includegraphics[width=2.5in]{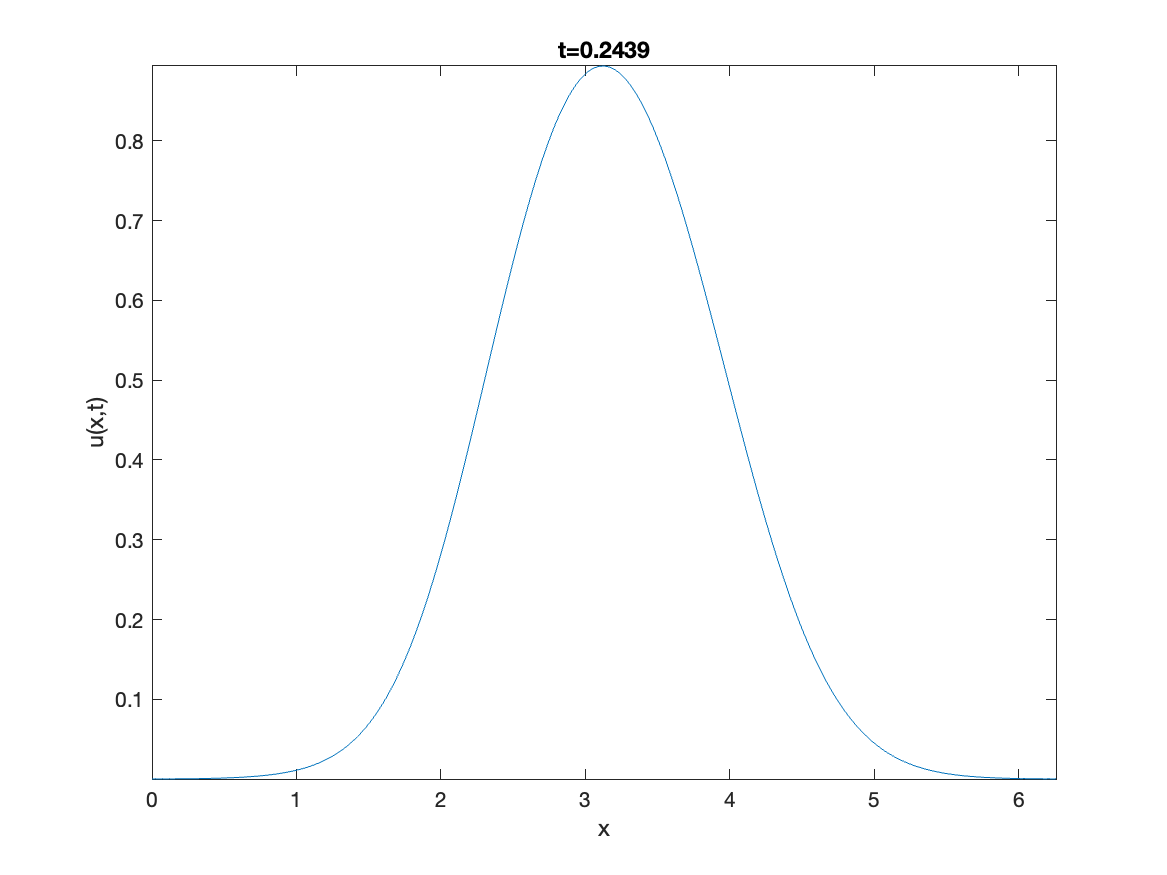}
\includegraphics[width=2.5in]{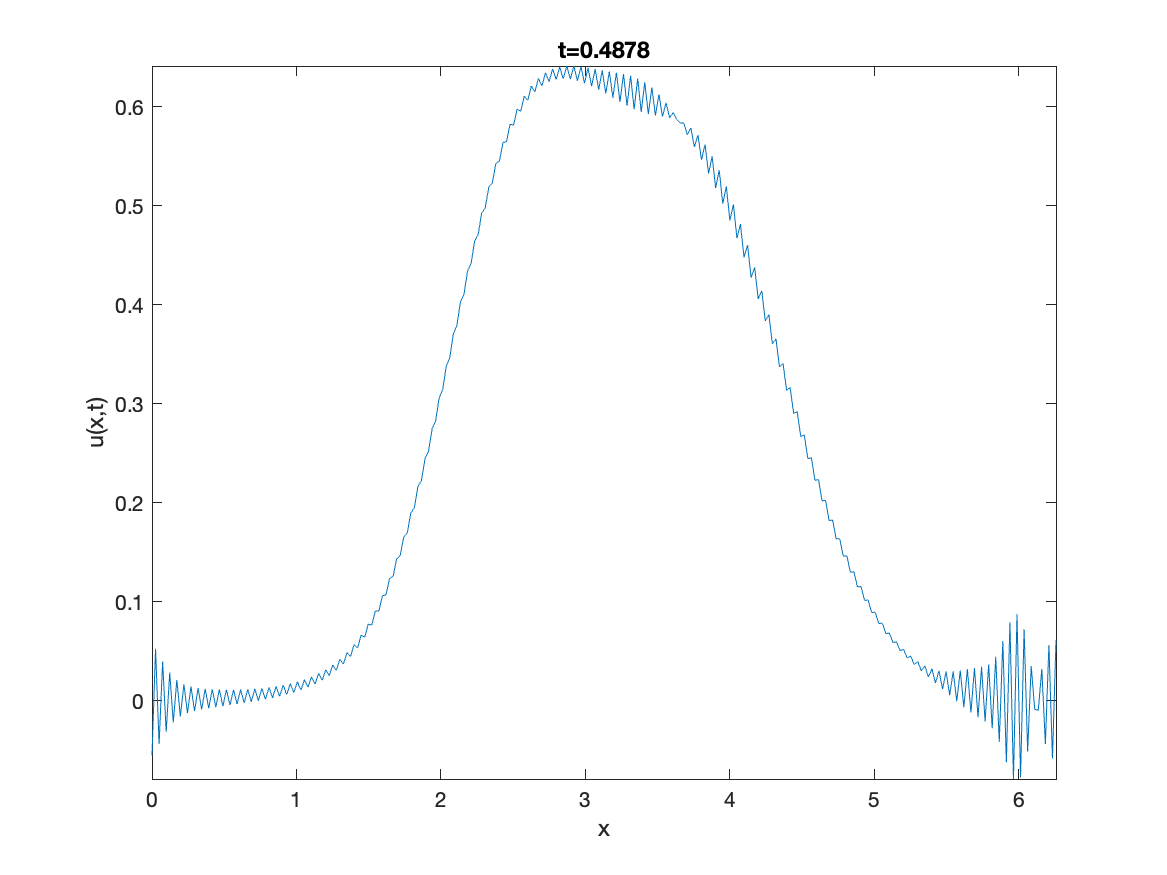} \\
\includegraphics[width=2.5in]{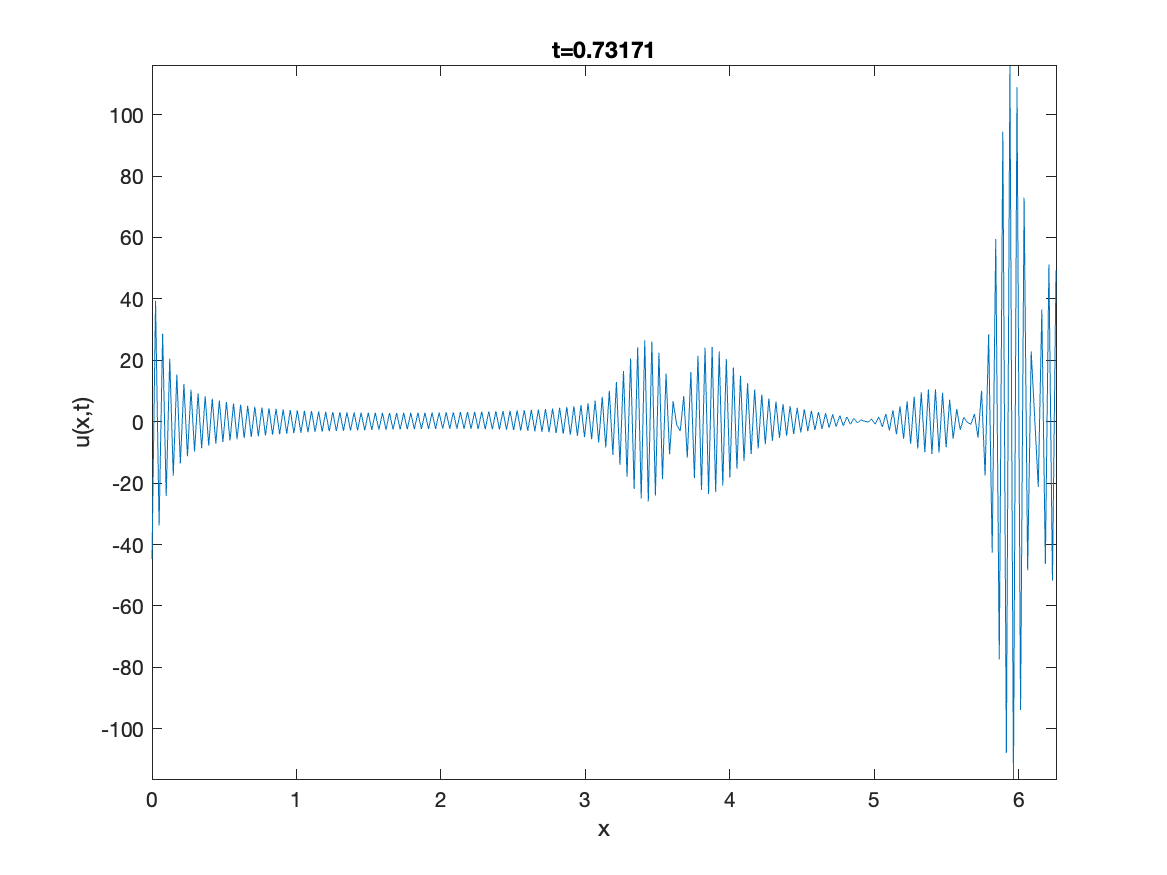}
\includegraphics[width=2.5in]{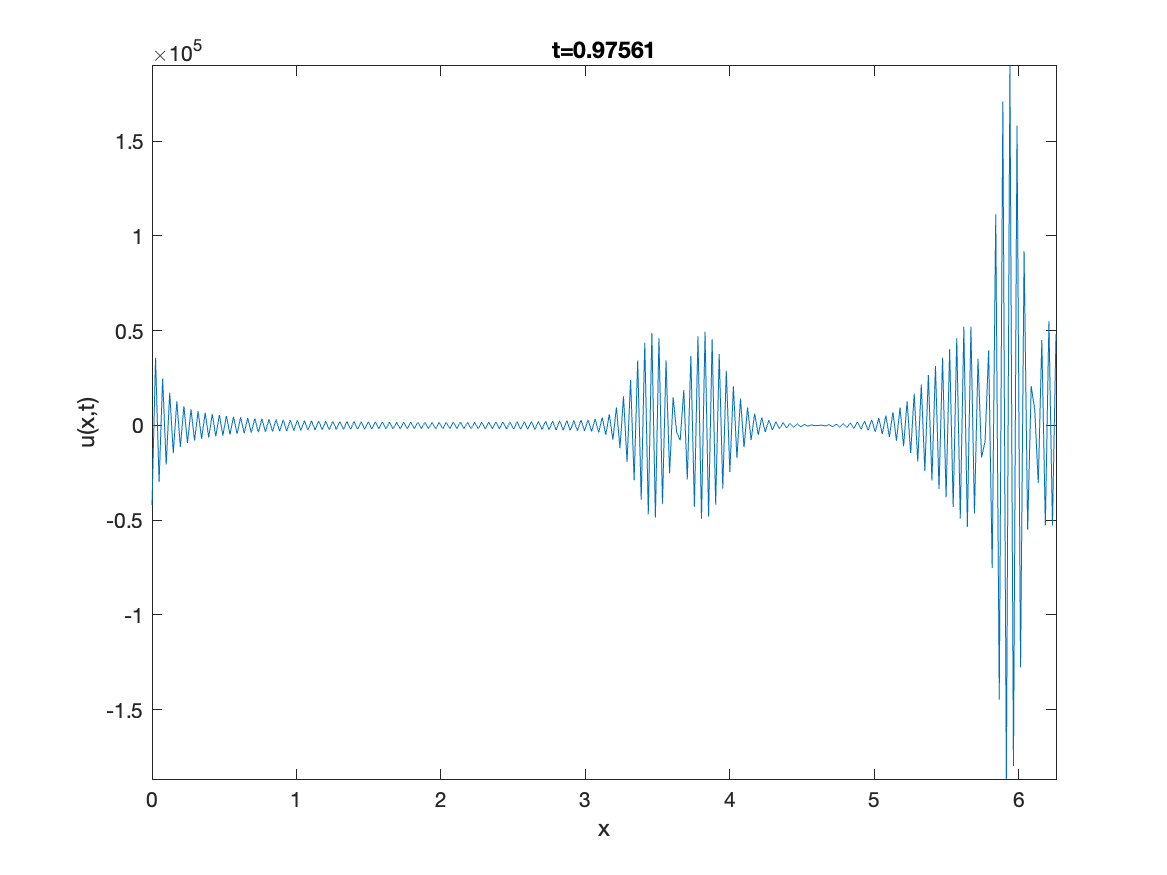}
\end{center}
\caption{Solutions of (\ref{eqnumexpPDE}), (\ref{eqnumexpP2}),  (\ref{eqnumexpQ}), (\ref{eqnumexpFG})
on the domain $(0,2\pi)\times(0,1)$, using the second-order KSS method described in
Section \ref{secconv}, with $N=256$ grid points and time step CFL number $\approx 1.74$.}
\label{figinstability}
\end{figure}

\begin{figure}[ht]
\begin{center}
\includegraphics[width=2.5in]{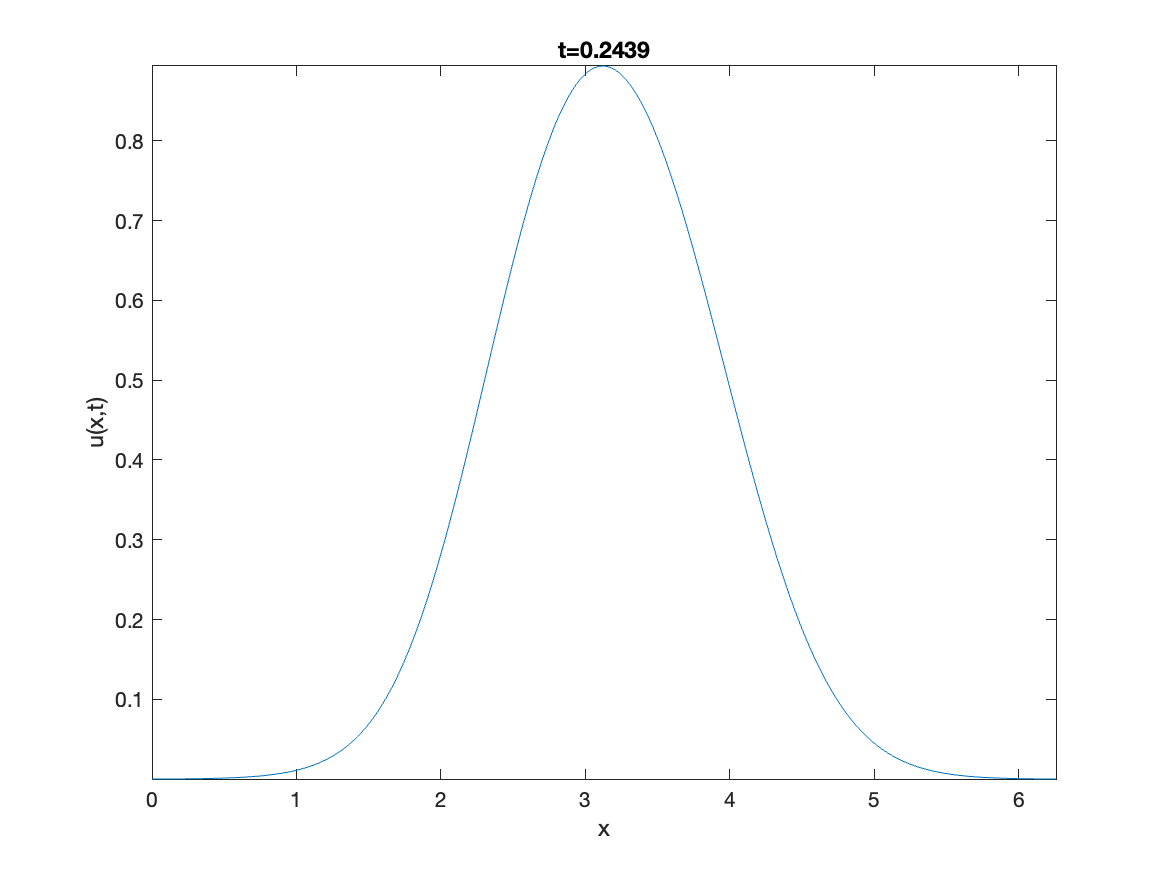}
\includegraphics[width=2.5in]{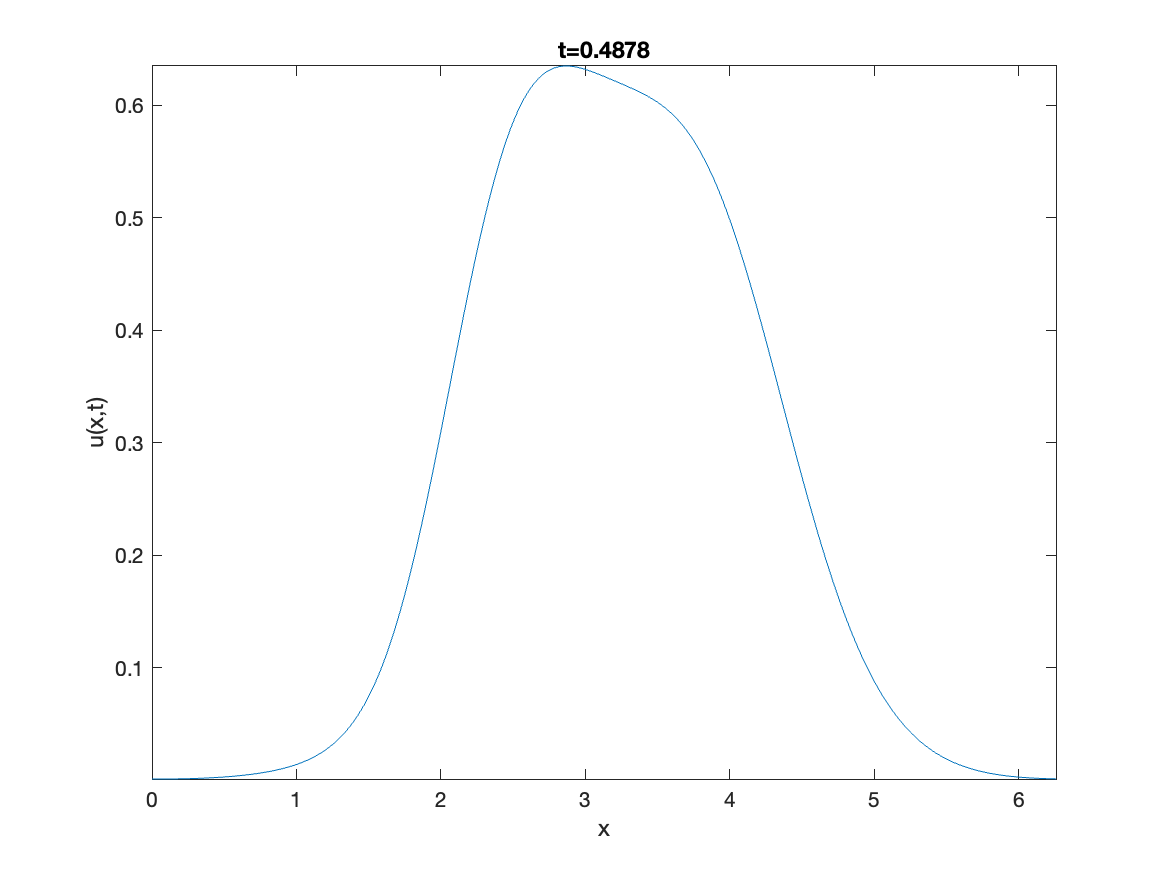} \\
\includegraphics[width=2.5in]{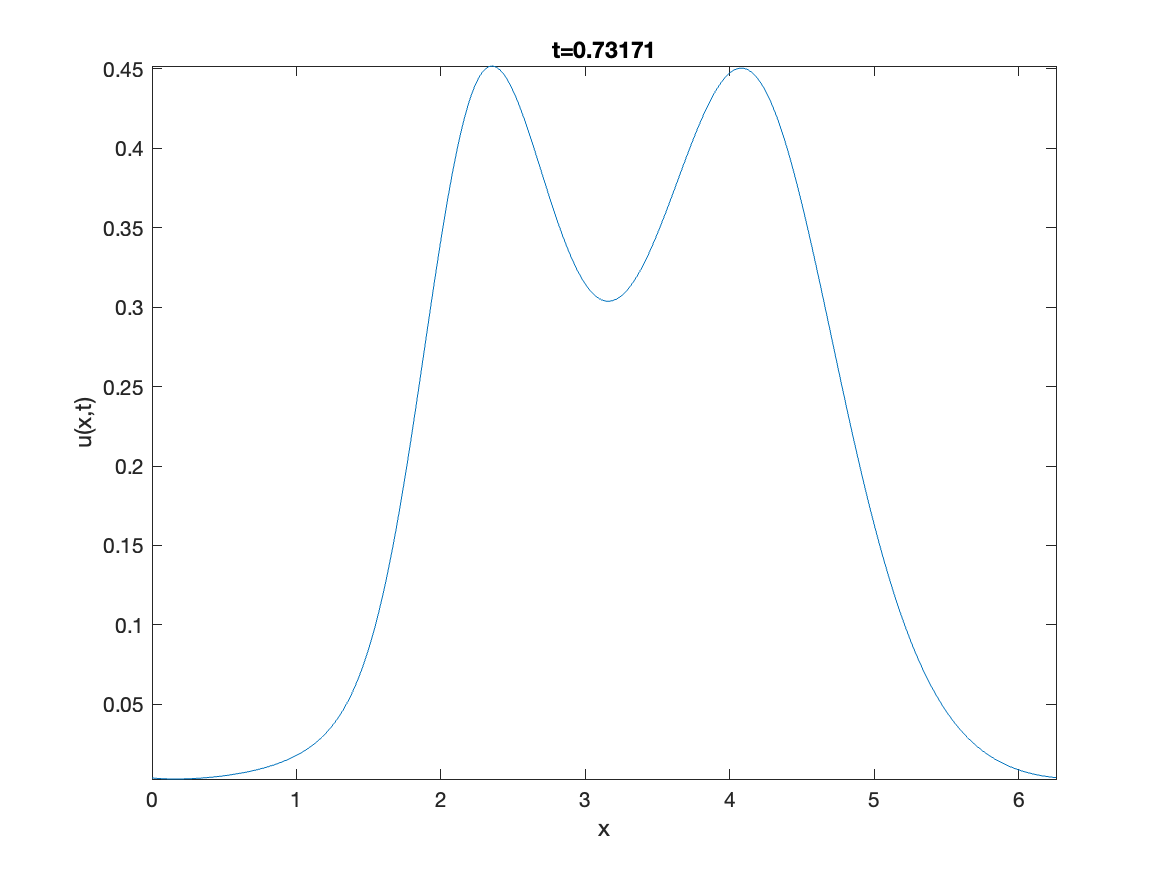}
\includegraphics[width=2.5in]{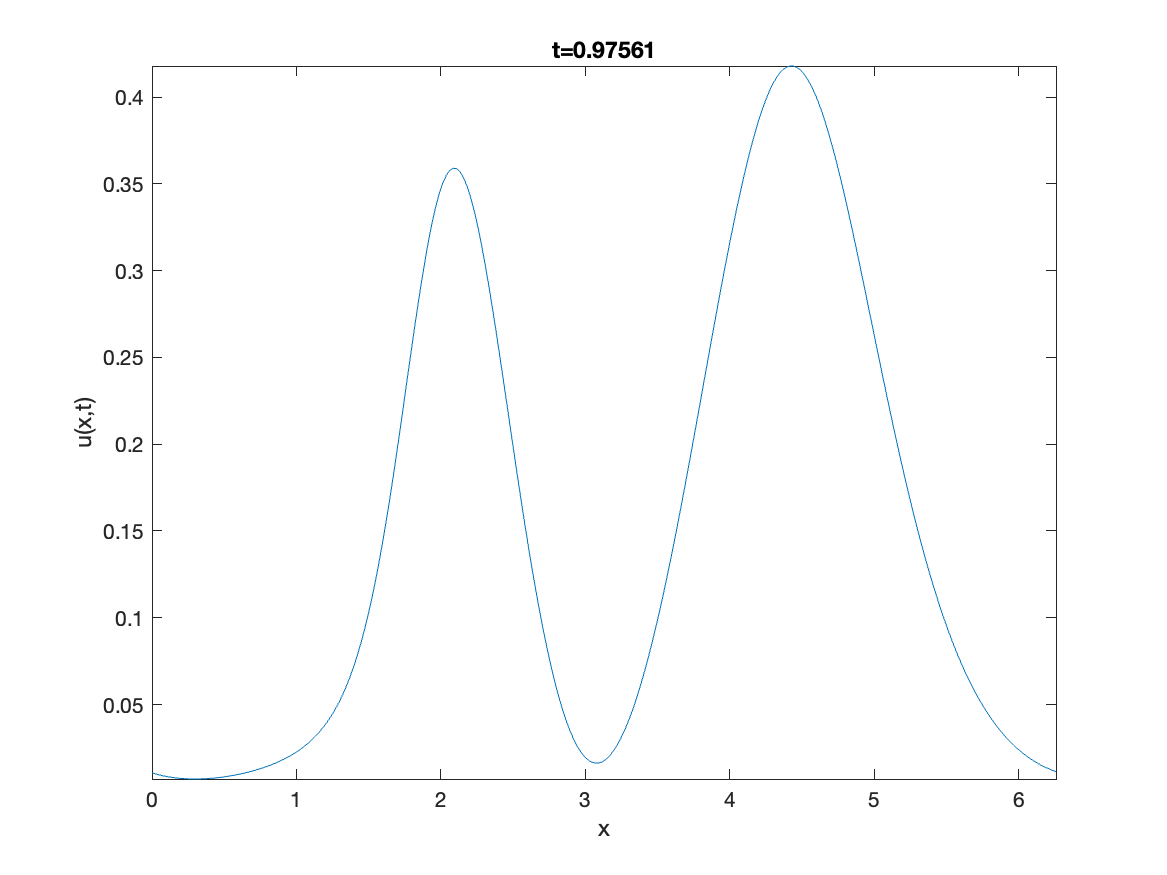}
\end{center}
\caption{Solutions of (\ref{eqnumexpPDE}), (\ref{eqnumexpP2}),  (\ref{eqnumexpQ}), (\ref{eqnumexpFG})
on the domain $(0,2\pi)\times(0,1)$, using the second-order KSS method described in
Section \ref{secconv}, with $N=256$ grid points and time step CFL number $\approx 0.87$.}
\label{figstability}
\end{figure}


The results of the experiments illustrated in Figures \ref{figinstability} and \ref{figstability}
are summarized in Table \ref{tabprob5}.  As $\Delta t$ is further decreased, the error
in the second-order KSS method decreases as $O(\Delta t^2)$.  The problem is also solved
with {\tt ode15s}, with its {\tt MaxStep} and {\tt InitialStep} parameters set to the value of each
time step $\Delta t$ used with KSS, to examine the behavior of the error as the maximum
time step approaches zero.  It is worth noting that {\tt ode15s} employs adaptive time-stepping,
while this implementation of KSS does not; adaptive time-stepping for KSS methods was
investigated in \cite{dozier}.  We observe that regardless of the maximum time step, {\tt ode15s}
produces a solution that is slightly more accurate than that of KSS, but KSS is significantly more 
efficient, as long as the (fixed) time step is chosen sufficiently small.

\begin{table}[ht]
\caption{}
\label{tabprob5}
\begin{center}
\footnotesize
\begin{tabular}{|c|r|r|r|r|r|} \hline
\multicolumn{2}{|c|}{} & \multicolumn{2}{|c|}{KSS} & \multicolumn{2}{|c|}{\tt ode15s} \\ \hline
$N$ & $\Delta t$ & error & time & error & time \\ \hline
 & $\pi/128$ &  -- & -- &  1.530e-05 & 1.284\\
256 & $\pi/256$ &  8.610e-05 & 0.006 &  1.282e-05 & 1.102\\
& $\pi/512$ &  1.941e-05 & 0.012 &  1.431e-05 & 1.099\\ \hline
\end{tabular}
\normalsize
\end{center}
\end{table}

\subsection{Generalizations}

In this paper, we have limited our analysis to the wave equation in one space dimension,
with periodic boundary conditions, and spatial differentiation performed using the FFT.
We now consider some variations of this problem, to investigate whether our conclusions
may apply more broadly.

\subsubsection{Finite Differencing in Space} \label{secFD}

We solve the problem (\ref{eqnumexpPDE}), (\ref{eqnumexpP}), (\ref{eqnumexpQ}),
(\ref{eqnumexpFG}), on the domain $(0,2\pi)\times(0,10)$,
with periodic boundary conditions, and using centered differencing
in space.  Because of the change of spatial discretization, we modify the interpolation
points from (\ref{eqinterpw}) by prescribing
\begin{equation}
\label{eqinterpwfd}
l_{2,\omega} = \overline{p} \frac{2-2\cos(\omega\Delta x)}{\Delta x^2} + \overline{q}, 
\quad \omega=-N/2+1,\ldots,N/2.
\end{equation}
The results are shown in Table \ref{tabT6}.  It can be seen that the same
unconditional stability that was established for spectral differentiation applies
in the case of finite differencing, as an accurate solution is obtained even when the
CFL number is as large as eight.

\begin{table}[ht]
\caption{Relative errors in the solution of
(\ref{eqnumexpPDE}), (\ref{eqnumexpP}), (\ref{eqnumexpQ}), (\ref{eqnumexpFG})
with periodic boundary conditions
on the domain $(0,2\pi)\times(0,10)$,
using the second-order KSS method described in Section \ref{secconv}, with $N$ grid points,
time step $\Delta t$, central differencing in space, and interpolation points
(\ref{eqinterpwfd}).}
\begin{center}
\footnotesize
\begin{tabular}{|l|r|r|r|r|} \hline
$\Delta t$ & $N=256$ & $N=512$ & $N=1024$ & $N=2048$ \\ \hline
$\pi/128$ & 1.00e-04 & 1.00e-04 & 1.00e-04 & 1.00e-04 \\ 
$\pi/256$ & 2.47e-05 & 2.48e-05 & 2.48e-05  & 2.47e-05\\
$\pi/512$ & 6.15e-06 & 6.16e-06 & 6.15e-06  & 6.15e-06 \\ \hline
\end{tabular}
\normalsize
\end{center}
\label{tabT6}
\end{table}

\subsubsection{Other Boundary Conditions} \label{secotherBC}

We repeat the problem from Section \ref{secFD}, except with homogeneous
Dirichlet boundary conditions.  The interpolation points from (\ref{eqinterpw}) are 
modified as follows:
\begin{equation}
\label{eqinterpwdir}
l_{2,\omega} = \overline{p} \frac{2-2\cos(\omega\Delta x/2)}{\Delta x^2} + \overline{q}, \quad
\omega = 0, 1, 2, \ldots, N-1.
\end{equation}
As in the case of periodic boundary conditions, unconditional stability is indicated
by the results, shown in Table \ref{tabT7}.

\begin{table}[ht]
\caption{Relative errors in the solution of
(\ref{eqnumexpPDE}), (\ref{eqnumexpP}), (\ref{eqnumexpQ}), (\ref{eqnumexpFG})
with homogeneous Dirichlet boundary conditions
on the domain $(0,2\pi)\times(0,10)$,
using the second-order KSS method described in Section \ref{secconv}, with $N$ grid points,
time step $\Delta t$, central differencing in space, and interpolation points
(\ref{eqinterpwdir}).}
\begin{center}
\footnotesize
\begin{tabular}{|l|r|r|r|r|} \hline
$\Delta t$ & $N=256$ & $N=512$ & $N=1024$ & $N=2048$ \\ \hline
$\pi/128$ & 1.53e-04 & 1.53e-04 & 1.53e-04 & 1.53e-04 \\ 
$\pi/256$ & 3.85e-05 & 3.85e-05 & 3.85e-05  & 3.86e-05\\
$\pi/512$ & 9.67e-06 & 9.67e-06 & 9.67e-06  & 9.68e-06 \\ \hline
\end{tabular}
\normalsize
\end{center}
\label{tabT7}
\end{table}

\subsubsection{Higher Space Dimension}

We solve a two-dimensional wave equation
\begin{equation} \label{eqnumexpPDE2D}
u_{tt} + Lu = 0, \quad 0 < x, y < 2\pi, \quad 0 < t < 10,
\end{equation}
where the differential operator $L$ is defined by
\begin{equation} \label{eqnumexpL2D}
Lu = -\Delta u + q(x,y)u,
\end{equation}
where
\begin{equation}
\label{eqnumexpQ2D}
q(x,y) = 1+ \frac{1}{2}\sin x\cos y+\frac{1}{4}\cos 2y +\frac{1}{8}\sin 3x.
\end{equation}
Our initial conditions are
\begin{equation} \label{eqnumexpFG2D}
u(x,y,0) = e^{-(x-\pi)^2+(y-\pi)^2}, \quad u_t(x,y,0) = 0, \quad 0 < x,y < 2\pi.
\end{equation}
and we impose periodic boundary conditions.
For spatial discretization, we use a $N\times N$ grid with spacing $\Delta x = \Delta y = 2\pi/N$,
and centered differencing in space.  This leads to the choice of interpolation points
\begin{equation} \label{eqinterpw2D}
l_{1,\omega_1,\omega_2} = 0, \quad l_{2,\omega_1,\omega_2}
= \frac{1}{\Delta x^2} [4  - 2\cos(\omega_1\Delta x) - 2\cos(\omega_2\Delta y)] + \overline{q},
\end{equation}
for $\omega_1,\omega_2 = -N/2+1,\ldots,N/2$, where $\overline{q}$ is the average value of $q(x,y)$
on $(0,2\pi)^2$.  The results, shown in Table \ref{tabT8}, indicate that unconditional stability
again holds, as the CFL number exceeds one without loss of accuracy or stability.

\begin{table}[ht]
\caption{Relative errors in the solution of
(\ref{eqnumexpPDE2D}), (\ref{eqnumexpQ2D}), (\ref{eqnumexpFG2D})
with periodic boundary conditions
on the domain $(0,2\pi)^2\times(0,10)$,
using the second-order KSS method described in Section \ref{secconv}, with $N$ grid points
per dimension,
time step $\Delta t$, central differencing in space, and interpolation points
(\ref{eqinterpw2D}).}
\begin{center}
\footnotesize
\begin{tabular}{|l|r|r|r|r|} \hline
$\Delta t$ & $N=16$ & $N=32$ & $N=64$ & $N=128$ \\ \hline
$\pi/8$ & 3.83e-02 & 4.14e-02 & 3.95e-02 & 3.84e-02 \\ 
$\pi/16$ & 8.69e-03 & 8.80e-03 & 8.30e-03  & 8.06e-03\\
$\pi/32$ & 2.01e-03 & 1.95e-03 & 1.84e-03  & 1.78e-03 \\ \hline
\end{tabular}
\normalsize
\end{center}
\label{tabT8}
\end{table}

\FloatBarrier

\subsection{Non-Bandlimited Coefficients}

We now carry out further investigation of the performance of KSS on problems beyond
those considered in the convergence analysis from Section \ref{secconv}.

We consider the problem (\ref{eqnumexpPDE}), (\ref{eqnumexpP}), 
(\ref{eqnumexpF}), (\ref{eqnumexpG}), with periodic boundary conditions, first with
\begin{equation} \label{eqnumexpQ1}
q(x) = \left\{ \begin{array}{ll} 1+\frac{1}{2}\sin x & 0 \leq x < \pi, \\
1-\frac{1}{2}\sin 2x & \pi \leq x < 2\pi,
\end{array} \right.
\end{equation}
which is constructed to as to be continuous but not differentiable at $x=\pi$, and then with
\begin{equation} \label{eqnumexpQ0}
q(x) = \left\{ \begin{array}{ll} 3/2 & 0 \leq x < \pi, \\
1/2 & \pi \leq x < 2\pi.
\end{array} \right.
\end{equation}
As shown in Tables \ref{tabprob10} and \ref{tabprob9}, the KSS method maintains stability and second-order accuracy in time.
\begin{table}[ht]
\caption{Relative errors in the solution of
(\ref{eqnumexpPDE}), (\ref{eqnumexpP}), (\ref{eqnumexpQ1}), (\ref{eqnumexpF}), (\ref{eqnumexpG})
with periodic boundary conditions
on the domain $(0,2\pi)\times(0,1)$,
using the second-order KSS method described in Section \ref{secconv}, with $N$ grid points,
time step $\Delta t$, and central differencing in space.}
\label{tabprob10}
\begin{center}
\footnotesize
\begin{tabular}{|c|r|r|r|r|} \hline
$\Delta t$ & $N=256$ & $N=512$ & $N=1024$ & $N=2048$ \\ \hline
$\pi/128$ & 5.331e-05 & 5.452e-05 & 5.342e-05 & 5.220e-05 \\ 
$\pi/256$ & 1.297e-05 & 1.336e-05 & 1.393e-05 & 1.381e-05 \\ 
$\pi/512$ & 3.219e-06 & 3.396e-06 & 3.597e-06 & 3.531e-06 \\  \hline
\end{tabular}
\normalsize
\end{center}
\end{table}
\begin{table}[ht]
\caption{Relative errors in the solution of
(\ref{eqnumexpPDE}), (\ref{eqnumexpP}), (\ref{eqnumexpQ0}), (\ref{eqnumexpF}), (\ref{eqnumexpG})
with periodic boundary conditions
on the domain $(0,2\pi)\times(0,1)$,
using the second-order KSS method described in Section \ref{secconv}, with $N$ grid points,
time step $\Delta t$, and central differencing in space.}
\label{tabprob9}
\begin{center}
\footnotesize
\begin{tabular}{|c|r|r|r|r|} \hline
$\Delta t$ & $N=256$ & $N=512$ & $N=1024$ & $N=2048$ \\ \hline
$\pi/128$ & 5.891e-05 & 6.005e-05 & 5.917e-05 & 5.783e-05 \\ 
$\pi/256$ & 1.417e-05 & 1.464e-05 & 1.533e-05 & 1.516e-05 \\ 
$\pi/512$ & 3.574e-06 & 3.758e-06 & 3.924e-06 & 3.870e-06 \\ \hline
\end{tabular}
\normalsize
\end{center}
\end{table}
Based on these numerical results, we seek to strengthen the result of Corollary \ref{corstab} by weakening
the assumption about the regularity of $q(x)$.
\begin{thm} \label{thm33new}
Assume $p(x)\equiv\textrm{constant}$, $q(x)$ is $2\pi$-periodic, and $q''(x)$ is piecewise $C^1$. Then the solution operator $S_N(\Delta t)$ satisfies 
\begin{eqnarray} \label{eqthm41}
\|S_N(\Delta t)\|_{C_N}\leq1+C_q\|\tilde{q}\|_\infty \Delta t,
\end{eqnarray}
where the constant $C_q$ is independent of $N$ and $\Delta t$.
\end{thm}

\begin{proof} The proof begins as in that of Theorem \ref{thm33} and its supporting lemmas.
We then have
\begin{eqnarray}
\|G_{11}\|_\infty
&\leq&\max\limits_{j\in\hat{I}_N} 1+ \frac{3}{2}\sum_{k\in\hat{I}_N\setminus j} \left| \hat q(k-j) (\bar{p}j^2+\bar{q})^{-1/2}(\bar{p}k^2+\bar{q})^{1/2} \Delta t^2 \right|   +  \nonumber \\
& & \frac{3}{2}\sum_{k\in\hat{I}_N\setminus j} \left| \hat q(k-j) (\bar{p}j^2+\bar{q})^{1/2}(\bar{p}k^2+\bar{q})^{-1/2} \Delta t^2 \right|   +  \nonumber \\
& &  \sum_{k\in\hat{I}_N} \left| (\bar{p}j^2+\bar{q})^{-1/2} (\bar{p}k^2+\bar{q})^{-1/2}\sum_{\omega\in\hat{I}_N\setminus\{k,j\}} |\hat q(\omega-k) \hat q(j-\omega)|\Delta t^2  \right|. \label{eqthm41G11}
\end{eqnarray}
By the assumptions on $q(x)$, it follows from \cite[Theorem A.1.3]{pdetxt} that there exists a constant
$C_q$ such that
$$|\hat{q}(\omega)| \leq \frac{C_q}{|\omega|^3+1}.$$
Therefore, if we define
$$C_0 = \sup_{\omega\in\mathbb{Z}\setminus\{0\}} \frac{|\hat{q}(\omega)(|\omega|^3+1)|}{\|\tilde{q}\|_\infty},$$
it follows that for $\omega\neq 0$,
\begin{equation} \label{eqqhatbd}
|\hat{q}(\omega)| \leq \frac{C_0 \|\tilde{q}\|_\infty}{|\omega|^3+1}.
\end{equation}
To bound the first summation in (\ref{eqthm41G11}), we consider
\begin{eqnarray}
\sum_{k\in\hat{I}_N\setminus j} \left| \hat q(k-j) \sqrt{\frac{\bar{p}k^2+\bar{q}}{\bar{p}j^2+\bar{q}}} \right|
& \leq & \frac{C_0\|\tilde{q}\|_\infty}{\sqrt{\bar{p}j^2+\bar{q}}} \sum_{k\in\hat{I}_N\setminus j} \frac{\sqrt{\bar{p}k^2+\bar{q}}}{|j-k|^3+1}  \nonumber \\
& \leq & \frac{C_0\|\tilde{q}\|_\infty}{\sqrt{\bar{p}j^2+\bar{q}}} \sum_{k\in\hat{I}_N\setminus j}\frac{\sqrt{\bar{p}}|k|}{|j-k|^3+1} + \frac{\sqrt{\bar{q}}}{|j-k|^3+1}. \label{eqthm41sum1}
\end{eqnarray}
From 
\begin{eqnarray*}
\sum_{k\in\hat{I}_N\setminus j} \frac{|k|}{|j-k|^3+1} & \leq &
\max\{2,|j|\} + \sum_{k\in\hat{I}_N, |k-j|>1} \frac{|k|}{|j-k|^3+1} \\
& \leq & \max\{2,|j|\} + 2 \sum_{u=2}^N \frac{u}{u^3+1} + \frac{|j|}{u^3+1} \\
& \leq & \max\{2,|j|\} + 2 \int_1^\infty u^{-2}  + |j| u^{-3}\,du\\
& \leq & \max\{2,|j|\} + 2 \left( 1 + \frac{|j|}{2} \right),
\end{eqnarray*}
we can conclude that the expression from (\ref{eqthm41sum1}) is bounded independently of $N$. 

Next, we show that the second summation from (\ref{eqthm41G11}),
\begin{eqnarray}
\sum_{k\in\hat{I}_N\setminus j} \left| \hat q(k-j) (\bar{p}j^2+\bar{q})^{1/2}(\bar{p}k^2+\bar{q})^{-1/2} \Delta t^2 \right|, \label{eqG11sum2}
\end{eqnarray}
can also be bounded independently of $N$.  Applying (\ref{eqqhatbd}), we focus on
\begin{eqnarray}
\sum_{k\in\hat{I}_N\setminus j} \left| \hat q(k-j) (\bar{p}k^2+\bar{q})^{-1/2} \right|
& \leq & \sum_{k\in\hat{I}_N\setminus j} \frac{(\bar{p}k^2+\bar{q})^{-1/2}}{|j-k|^3+1}\nonumber \\
& \leq & \frac{1}{(|j|^3 + 1)\bar{q}} + \frac{1}{\sqrt{\bar{p}}} \sum_{k\in\hat{I}_N\setminus \{0,j\}} \frac{1}{|j-k|^3|k|}.\label{eqsum2orig}
\end{eqnarray}
If $j=0$, then we have
$$
\sum_{k\in\hat{I}_N\setminus 0} \frac{1}{ |k|^3 |k|} 
\leq  2\left( \int_{1}^{\infty} \frac{1}{k^4} \,  dk +1 \right) 
 \leq  \frac{8}{3}.
$$
If $j>0$, then we bound the final summation in (\ref{eqsum2orig}) as follows:
\begin{eqnarray*}
\sum_{k=-\frac{N}{2}+1}^{-1} \frac{1}{(j-k)^3(-k)} 
& \leq & \frac{1}{(j+1)^3} + \int_{-\infty}^{-1} \frac{1}{(j-k)^3(-k)}\,dk \\
& \leq & \frac{1}{(j+1)^3} - \frac{3j+2}{2j^2(j+1)^2} + \frac{\ln|j+1|}{j^3},\\
\sum_{k=1}^{j-1} \frac{1}{(j-k)^3 k} 
& \leq & \frac{1}{j-1}+ \frac{1}{(j-1)^3} + \int_{1}^{j-1} \frac{1}{(j-k)^3k} \,dk \\
& \leq & \frac{1}{j-1} + \frac{1}{(j-1)^3}+ \frac{j^3 - 6j + 4}{2j^2(j-1)^2} + \frac{2 \ln |j-1|}{j^3},\\
\sum_{k=j+1}^{N/2} \frac{1}{(k-j)^3 k} 
& \leq & \frac{1}{(j+1)} + \int_{j+1}^{\infty} \frac{1}{(k-j)^3k} \,dk \\
& \leq & \frac{1}{(j+1)} + \frac{\ln |j+1|}{j^3} - \frac{1}{j^2} + \frac{1}{2j}.
\end{eqnarray*}
As all of these portions of (\ref{eqsum2orig}) are $O(j^{-1})$, we find that
(\ref{eqG11sum2}) is bounded independently of $N$.
%

Finally, we consider the third summation from (\ref{eqthm41G11}),
\begin{eqnarray} \label{eq3rdsum}
 \sum_{k\in\hat{I}_N} \left| (\bar{p}j^2+\bar{q})^{-1/2} (\bar{p}k^2+\bar{q})^{-1/2}\sum_{\omega\in\hat{I}_N\setminus\{k,j\}} |\hat q(\omega-k) \hat q(j-\omega)|\Delta t^2  \right|.
\end{eqnarray}
Applying (\ref{eqqhatbd}) to the sum over $\omega$, we then focus on bounding
\begin{equation} \label{eq3rdinnersum}
\sum_{\omega\in\hat{I}_N\setminus\{k,j\}} \frac{1}{|\omega-k|^3|j-\omega|^3}.
\end{equation}
Let $z \equiv k-j>1$, and let $m=\lfloor (j+k)/2 \rfloor$.  We then have
\begin{eqnarray*}
\sum_{\omega\in\hat{I}_N\setminus\{k,j\}} \frac{1}{|\omega-k|^3|j-\omega|^3} \nonumber 
& \leq &   \frac{2}{|j-1-k|^3} + \frac{2}{|j+1-k|^3} + \int_{-\infty}^{j-1} \frac{1}{[(k-\omega)(j-\omega)]^3} \,d\omega + \nonumber \\ & &
\int_{j+1}^{m} \frac{1}{[(\omega-k)(j-\omega)]^3} \,d\omega + \int_{m+1}^{k-1} \frac{1}{[(\omega-k)(j-\omega)]^3} \,d\omega + \nonumber \\ & &
\int_{k+1}^{\infty} \frac{1}{[(\omega-k)(\omega-j)]^3} \,d\omega     \\
 & \leq & \frac{2}{|z+1|^3} + \frac{2}{|z-1|^3} + 
 \frac{2z^2 - \frac{2z^2}{(z-1)^2}  - \frac{12z}{z-1}  + 24 \ln |z-1|}{2z^5} + \nonumber \\ & &
 \frac {\frac{z^2}{(\frac{-z}{2}-1)^2} - \frac{z^2}{(\frac{z}{2} - 1)^2}  - \frac{6z}{(\frac{-z}{2} - 1)} -\frac{6z}{(\frac{z}{2} - 1)}  - 12 \ln |\frac{z}{2}+ 1| + 12 \ln |\frac{-z}{2}+1|}{2z^5}\\
& \leq & \tilde{C}z^{-3}
\end{eqnarray*}
for some constant $\tilde{C}$.
By symmetry, the case of $z<1$ is identical, and by direct evaluation, the terms corresponding
to $|z|\leq 1$ are bounded independently of $j$.  Summing the bounds on
(\ref{eq3rdinnersum}) over all $z\in\mathbb{Z}$,
we conclude that (\ref{eq3rdsum}) is bounded independently of $N$ and is $O(\Delta t^2)$.

In summary, there exist constants $C_{11,q}$ and $C_{11,q^2}$, independent of $N$ and
$\Delta t$, such that
\begin{eqnarray*}
\|G_{11}\|_\infty &\leq&  1 +  C_{11,q} \|\tilde{q}\|_{\infty} \Delta t^2
 + C_{11,q^2} \|\tilde{q}\|_{\infty}^2 \Delta t^2.
\end{eqnarray*}
Using the same approach, we find that the matrix $G_{22}$ defined in (\ref{G}), with
the assumption that $p(x)$ is constant, satisfies a bound of the
same form as that of $G_{11}$, with appropriate constant factors. 

Proceeding as in the proof of Lemma \ref{lem32}, we have
\begin{eqnarray*}
\|G_{12}\|_{\infty} 
&\leq&  \max\limits_{1\leq j\leq N} 6\sum_{k\in\hat{I}_N\setminus j} \left|  \hat q(k-j) (\bar{p}j^2+\bar{q})^{-1/2} \Delta t \right| + \\
& & \sum_{k\in\hat{I}_N\setminus j} \left|  \sum_{\omega\in\hat{I}_N\setminus\{k,j\}} |\hat q(j-\omega) \hat q(\omega-k) |(\bar{p}j^2+\bar{q})^{-1/2} \frac{4\Delta t}{\overline{p}\omega^2+\overline{q}} \right|. 
\end{eqnarray*}
These summations can be bounded using the same approach as in the case of $G_{11}$, and therefore
there exist constants $C_{12,q}$ and $C_{12,q^2}$ independent of $N$ and $\Delta t$ such that
\begin{eqnarray*}
\|G_{12}\|_{\infty} &\leq& C_{12,q}  \|\tilde{q}\|_{\infty} \Delta t  + C_{12,q^2} \|\tilde{q}\|_{\infty}^2 \Delta t.
\end{eqnarray*}
The same approach can also be used to show that 
the matrix $G_{21}$ in (\ref{G}), under the assumption that $p(x)$ is constant, 
satisfies a bound of the same form as that of $G_{12}$.

Proceeding as in the proof of Theorem \ref{thm33} yields (\ref{eqthm41}).
\end{proof}

\vspace{0.1in}
\noindent We now present numerical evidence that the conclusion of Theorem \ref{thm33new}
holds under an even weaker assumption about the regularity of $q(x)$.  Table \ref{tabnorms}
shows that for the differential operator (\ref{eqLform}), with $p(x)$ constant and $q(x)$
piecewise constant, $\|S_N(\Delta t)\|_{C_N}$ appears to be bounded independently of $N$.  Unfortunately, such a bound cannot be proved using the same approach as in the proof of Theorem \ref{thm33new}, as the upper bound established is not sufficiently sharp.

\begin{table}[ht]
\caption{$\|S_N(\Delta t)\|_{C_N}$ for various values of $N$ and $\Delta t$, where
$S_N(\Delta t)$, as defined in (\ref{eqSNdef}), is the approximate solution operator for the
KSS method described in Section \ref{secconv} for the problem (\ref{eqnumexpPDE}), (\ref{eqnumexpP}), 
(\ref{eqnumexpQ0}), (\ref{eqnumexpF}), (\ref{eqnumexpG}).}
\label{tabnorms}
\begin{center}
\begin{tabular}{|l|r|r|r|} \hline
$\Delta t$ & $N=256$ & $N=512$ & $N=1024$ \\ \hline
       1 &   1.272444 &  1.272439 &  1.272438 \\
    0.1 &   1.025053 &  1.025047 &  1.025045 \\
  0.01 &   1.002502  & 1.002502 &  1.002501 \\
0.001 &   1.000250 &  1.000250 &  1.000250 \\ \hline
\end{tabular}
\end{center}
\end{table}


\FloatBarrier

\subsection{Comparison with Krylov Solvers}

We will now compare the performance of our KSS method with an implicit time-stepping method,
in which a Krylov subspace method is used to solve systems of linear equations.  We consider
the problem (\ref{eqnumexpPDE}), (\ref{eqnumexpP}), (\ref{eqnumexpQ1}), (\ref{eqnumexpF}), (\ref{eqnumexpG}), with periodic boundary conditions.

After spatial discretization, we solve the system of ODEs (\ref{eqwaveodesys}) using the trapezoidal
rule for time-stepping, which requires solving the systems of linear equations
\begin{equation} \label{eqgmressys}
\left(I_{2N} - \frac{\Delta t}{2} \tilde{L}_N \right){\bf v}_N^{n+1} = 
\left(I_{2N} + \frac{\Delta t}{2} \tilde{L}_N \right){\bf v}_N^{n}, \quad n = 0, 1, 2, \ldots.
\end{equation}
To solve each system, we use GMRES, with ILU(0) preconditioning \cite{ilu}.  The results are
shown in Table \ref{tabprob10k}.  We observe that the trapezoidal rule is not nearly as accurate
as KSS, even when the system (\ref{eqgmressys}) is solved to very high accuracy.
Furthermore, the accuracy deteriorates as the grid size increases, and second-order accuracy
is not maintained.  This is due to the fact that the initial data, and therefore the solution, is not 
smooth; with smooth initial data, the trapezoidal rule is more accurate, and does not lose accuracy
as $N$ increases, though it is still not as accurate as KSS.

Finally, the number of iterates needed by GMRES for convergence, 
though reduced to some extent by the preconditioning, still increases
with both $N$ and $\Delta t$ (whether the initial data is smooth or not), 
while the number of FFTs or matrix-vector multiplications required
by KSS are not influenced by these parameters.  Similar results were obtained
when using BiCGSTAB in place of GMRES, except that, on average, even more iterations
were required for convergence.

\begin{table}[ht]
\caption{Relative errors, execution times in seconds, and average iteration counts in the solution of
(\ref{eqnumexpPDE}), (\ref{eqnumexpP}), (\ref{eqnumexpQ1}), (\ref{eqnumexpF}), (\ref{eqnumexpG})
with periodic boundary conditions
on the domain $(0,2\pi)\times(0,1)$,
using the second-order KSS method described in Section \ref{secconv} (labeled ``KSS'')
and the trapezoidal rule with GMRES and ILU preconditioning (labeled ``GMRES'').
Both methods use $N$ grid points, time step $\Delta t$, and central differencing in space.}
\label{tabprob10k}
\begin{center}
\footnotesize
\begin{tabular}{|c|r|r|r|r|r|r|} \hline
\multicolumn{2}{|c|}{} & \multicolumn{2}{|c|}{KSS} & \multicolumn{3}{|c|}{GMRES} \\ \hline
$N$ & $\Delta t$ & error & time & error & time & iter. \\ \hline
 & $\pi/64$ &  2.313e-04 & 0.001 &  4.685e-02 & 0.009 & 8\\
 & $\pi/128$ &  5.891e-05 & 0.001 &  2.091e-02 & 0.013 & 5\\
256 & $\pi/256$ &  1.417e-05 & 0.002 &  1.080e-02 & 0.024 & 4\\
 & $\pi/512$ &  3.574e-06 & 0.004 &  2.979e-03 & 0.044 & 3\\
 & $\pi/1024$ &  8.958e-07 & 0.008 &  7.153e-04 & 0.087 & 3\\
 & $\pi/2048$ &  2.242e-07 & 0.015 &  1.844e-04 & 0.176 & 3\\ \hline
 & $\pi/64$ &  2.203e-04 & 0.001 &  3.714e-02 & 0.022 & 13\\
 & $\pi/128$ &  6.005e-05 & 0.002 &  2.574e-02 & 0.032 & 9\\
512 & $\pi/256$ &  1.464e-05 & 0.004 &  1.210e-02 & 0.050 & 6\\
 & $\pi/512$ &  3.758e-06 & 0.007 &  4.789e-03 & 0.084 & 4\\
 & $\pi/1024$ &  9.518e-07 & 0.015 &  1.921e-03 & 0.145 & 3\\
 & $\pi/2048$ &  2.393e-07 & 0.029 &  5.675e-04 & 0.292 & 3\\ \hline
 & $\pi/64$ &  2.112e-04 & 0.001 &  4.279e-02 & 0.073 & 22\\
 & $\pi/128$ &  5.917e-05 & 0.002 &  2.155e-02 & 0.112 & 16\\
1024 & $\pi/256$ &  1.533e-05 & 0.004 &  1.508e-02 & 0.177 & 10\\
 & $\pi/512$ &  3.924e-06 & 0.007 &  7.994e-03 & 0.318 & 7\\
 & $\pi/1024$ &  9.764e-07 & 0.015 &  3.316e-03 & 0.585 & 5\\
 & $\pi/2048$ &  2.428e-07 & 0.031 &  1.304e-03 & 1.134 & 4\\ \hline
\end{tabular}
\normalsize
\end{center}
\end{table}

\FloatBarrier

\subsection{Comparison with Exponential Integrators}

Next, we apply our KSS method to a nonlinear problem, and compare its performance
to that of exponential integrators that employ Krylov subspace methods to evaluate matrix function-vector products.  We consider the Klein-Gordon equation \cite{kge}
\begin{equation} \label{eqKGE}
u_{tt} = u_{xx} - u^3, \quad 0 < x < 2\pi, \quad t > 0,
\end{equation}
with initial conditions (\ref{eqnumexpF}), (\ref{eqnumexpG}), and periodic boundary conditions.
The second-order KSS method described in Section \ref{secconv} is compared with the following
methods:
\begin{itemize}
\item A Gautschi-type method presented in \cite{grimm,hochlub}, in which matrix function-vector 
products are computed by applying Lanczos iteration, as described in \cite{hochlubsel}.
This method will be referred to as ``Gautschi-Krylov''.
\item The exponential Euler method \cite{ostermann22}, with matrix function-vector products 
computed using an adaptive Krylov iteration from \cite{adapkrylov}.  This method will be referred
to as ``adaptive Krylov''.
\end{itemize}
The results are shown in Table \ref{tabprob12}.  For Gautschi-Krylov and adaptive Krylov,
the iteration counts reported in the table refer to the average number of matrix-vector multiplications
performed in the course of approximating matrix function-vector products.  For all three methods,
the following computational expense is incurred during each time step:
\begin{itemize}
\item For KSS, three matrix-vector multiplications, three FFTs, and two inverse FFTs, in the course
of approximating four matrix function-vector products, with an $N\times N$ matrix.
\item For Gautschi-Krylov, two matrix function-vector products, each involving, on average,
the number of matrix-vector multiplications reported in Table \ref{tabprob12}, with an
$N\times N$ matrix.
\item For adaptive Krylov, one matrix function-vector product, involving, on average,
the number of matrix-vector multiplications reported in Table \ref{tabprob12}, with a $2N\times 2N$ matrix.
\end{itemize}
As can be seen in the table, the number of matrix function-vector products required by
Gautschi-Krylov and adaptive Krylov increases with $N$ and $\Delta t$.  The accuracy of KSS
and adaptive Krylov is comparable, while Gautschi-Krylov is somewhat more accurate than both,
but KSS is significantly faster than both.
\begin{table}[ht]
\caption{Relative errors, execution times in seconds, and iteration counts in the solution of (\ref{eqKGE}), 
(\ref{eqnumexpF}), (\ref{eqnumexpG}) with periodic boundary conditions on the domain
$(0,2\pi)\times(0,1)$, using the
second-order KSS method described in Section \ref{secconv} (labeled ``KSS''),
the Gautschi-type method from \cite{grimm,hochlub} (labeled ``Gautschi-Krylov''),
and the exponential Euler method \cite{ostermann22} with adaptive Krylov
iteration \cite{adapkrylov} (labeled ``Adaptive Krylov'').  
All methods use $N$ grid points, time step $\Delta t$, 
and central differencing in space.}
\label{tabprob12}
\begin{center}
\footnotesize
\begin{tabular}{|c|r|r|r|r|r|r|r|r|r|} \hline
\multicolumn{2}{|c|}{} & \multicolumn{2}{|c|}{KSS} & \multicolumn{3}{|c|}{Gautschi-Krylov} & \multicolumn{3}{|c|}{Adaptive Krylov} \\ \hline
$N$ & $\Delta t$ & error & time & error & time & iter. & error & time & iter. \\ \hline
 & $\pi/64$ &  5.441e-04 & 0.001 &  1.796e-04 & 0.013 & 8 &  4.967e-04 & 0.037 & 16\\
 & $\pi/128$ &  1.353e-04 & 0.001 &  4.309e-05 & 0.015 & 6 &  1.315e-04 & 0.049 & 11\\
256 & $\pi/256$ &  3.328e-05 & 0.002 &  1.036e-05 & 0.024 & 5 &  3.304e-05 & 0.084 & 8\\
 & $\pi/512$ &  8.430e-06 & 0.004 &  2.533e-06 & 0.033 & 4 &  8.384e-06 & 0.161 & 7\\
 & $\pi/1024$ &  2.105e-06 & 0.008 &  6.346e-07 & 0.055 & 4 &  2.099e-06 & 0.277 & 6\\
 & $\pi/2048$ &  5.258e-07 & 0.016 &  1.588e-07 & 0.095 & 3 &  5.248e-07 & 0.542 & 5\\ \hline
 & $\pi/64$ &  5.221e-04 & 0.001 &  1.616e-04 & 0.025 & 12 &  4.919e-04 & 0.059 & 21\\
 & $\pi/128$ &  1.376e-04 & 0.002 &  4.434e-05 & 0.028 & 9 &  1.307e-04 & 0.074 & 15\\
512 & $\pi/256$ &  3.351e-05 & 0.003 &  1.068e-05 & 0.036 & 7 &  3.285e-05 & 0.109 & 11\\
 & $\pi/512$ &  8.375e-06 & 0.006 &  2.600e-06 & 0.059 & 6 &  8.331e-06 & 0.191 & 8\\
 & $\pi/1024$ &  2.090e-06 & 0.012 &  6.228e-07 & 0.090 & 5 &  2.085e-06 & 0.360 & 7\\
 & $\pi/2048$ &  5.223e-07 & 0.025 &  1.575e-07 & 0.148 & 4 &  5.216e-07 & 0.645 & 6\\ \hline
 & $\pi/64$ &  5.158e-04 & 0.002 &  1.614e-04 & 0.110 & 18 &  4.893e-04 & 0.099 & 35\\
 & $\pi/128$ &  1.348e-04 & 0.004 &  4.103e-05 & 0.073 & 12 &  1.299e-04 & 0.137 & 22\\
1024 & $\pi/256$ &  3.358e-05 & 0.008 &  1.066e-05 & 0.075 & 9 &  3.267e-05 & 0.192 & 15\\
 & $\pi/512$ &  8.423e-06 & 0.015 &  2.663e-06 & 0.089 & 7 &  8.293e-06 & 0.299 & 11\\
 & $\pi/1024$ &  2.087e-06 & 0.031 &  6.429e-07 & 0.138 & 6 &  2.076e-06 & 0.502 & 9\\
 & $\pi/2048$ &  5.201e-07 & 0.063 &  1.512e-07 & 0.228 & 5 &  5.191e-07 & 0.866 & 7\\ \hline
\end{tabular}
\end{center}
\end{table}

Next, we consider another Klein-Gordon equation,
\begin{equation} \label{eqKGEv}
u_{tt} = (p(x)u_x)_x - q(x)u-u^3, \quad 0 < x < 2\pi, \quad t > 0,
\end{equation}
with $p(x)$ from (\ref{eqnumexpP2}), $q(x)$ from (\ref{eqnumexpQ}), 
initial conditions (\ref{eqnumexpF}), (\ref{eqnumexpG}), and periodic boundary conditions.
The results are shown in Table \ref{tabprob14}.  We see that the KSS method cannot produce
an accurate solution when $\Delta t > \Delta x$; the method is unstable in this case, due to
$p(x)$ varying with $x$.  For $\Delta t$ sufficiently small, KSS exhibits second-order accuracy,
and accuracy comparable to that of adaptive Krylov.  Gautschi-Krylov is the most accurate method
of the three, but KSS, when stable, is able to deliver greater accuracy in less time.
\begin{table}[ht]
\caption{Relative errors, execution times in seconds, and iteration counts in the solution of (\ref{eqKGEv}), 
(\ref{eqnumexpF}), (\ref{eqnumexpG}) with periodic boundary conditions on the domain
$(0,2\pi)\times(0,1)$, using the
second-order KSS method described in Section \ref{secconv} (labeled ``KSS''),
the Gautschi-type method from \cite{grimm,hochlub} (labeled ``Gautschi-Krylov''),
and the exponential Euler method \cite{ostermann22} with adaptive Krylov
iteration \cite{adapkrylov} (labeled ``Adaptive Krylov'').  
All methods use $N$ grid points, time step $\Delta t$, 
and central differencing in space.}
\label{tabprob14}
\small
\begin{center}
\footnotesize
\begin{tabular}{|c|r|r|r|r|r|r|r|r|r|} \hline
\multicolumn{2}{|c|}{} & \multicolumn{2}{|c|}{KSS} & \multicolumn{3}{|c|}{Gautschi-Krylov} & \multicolumn{3}{|c|}{Adaptive Krylov} \\ \hline
$N$ & $\Delta t$ & error & time & error & time & iter. & error & time & iter. \\ \hline
 & $\pi/64$ &  -- & -- &  4.011e-04 & 0.008 & 7 &  1.076e-03 & 0.025 & 11\\
 & $\pi/128$ &  2.887e-04 & 0.001 &  1.079e-04 & 0.010 & 5 &  2.817e-04 & 0.042 & 8\\
256 & $\pi/256$ &  7.117e-05 & 0.003 &  2.738e-05 & 0.014 & 4 &  7.031e-05 & 0.071 & 6\\
 & $\pi/512$ &  1.788e-05 & 0.005 &  7.010e-06 & 0.025 & 4 &  1.778e-05 & 0.137 & 5\\
 & $\pi/1024$ &  4.455e-06 & 0.011 &  1.755e-06 & 0.043 & 3 &  4.442e-06 & 0.224 & 4\\
 & $\pi/2048$ &  1.112e-06 & 0.021 &  4.391e-07 & 0.087 & 3 &  1.110e-06 & 0.352 & 3\\ \hline
 & $\pi/64$ &  -- & -- &  4.010e-04 & 0.019 & 10 &  1.076e-03 & 0.051 & 20\\
 & $\pi/128$ &  -- & -- &  1.079e-04 & 0.019 & 7 &  2.816e-04 & 0.055 & 11\\
512 & $\pi/256$ &  7.113e-05 & 0.003 &  2.740e-05 & 0.029 & 6 &  7.027e-05 & 0.093 & 8\\
 & $\pi/512$ &  1.788e-05 & 0.007 &  6.978e-06 & 0.041 & 4 &  1.777e-05 & 0.158 & 6\\
 & $\pi/1024$ &  4.453e-06 & 0.014 &  1.755e-06 & 0.068 & 4 &  4.440e-06 & 0.308 & 5\\
 & $\pi/2048$ &  1.111e-06 & 0.028 &  4.390e-07 & 0.113 & 3 &  1.110e-06 & 0.513 & 4\\ \hline
 & $\pi/64$ &  -- & -- &  4.010e-04 & 0.089 & 15 &  1.075e-03 & 0.100 & 55\\
 & $\pi/128$ &  -- & -- &  1.079e-04 & 0.055 & 11 &  2.815e-04 & 0.120 & 19\\
1024 & $\pi/256$ &  -- & -- &  2.737e-05 & 0.051 & 8 &  7.026e-05 & 0.167 & 12\\
 & $\pi/512$ &  1.787e-05 & 0.012 &  6.981e-06 & 0.068 & 6 &  1.776e-05 & 0.215 & 8\\
 & $\pi/1024$ &  4.452e-06 & 0.025 &  1.751e-06 & 0.104 & 5 &  4.439e-06 & 0.361 & 6\\
 & $\pi/2048$ &  1.111e-06 & 0.052 &  4.389e-07 & 0.160 & 4 &  1.109e-06 & 0.699 & 5\\ \hline
\end{tabular}
\normalsize
\end{center}
\end{table}

Finally, we compare Gautschi-Krylov to a variation of Gautschi-Krylov in which any matrix
function-vector products are computed using KSS; this variation is referred to as ``Gautschi-KSS''.  
As can be seen in Table \ref{tabprob12g},
this variation combines the greater stability
of Gautschi-Krylov with the efficiency and scalability of KSS.  Gautschi-KSS is significantly
faster than KSS alone (and therefore has an even greater advantage in terms of efficiency over Gautschi-Krylov), and is not unstable for $\Delta t > \Delta x$.  For larger
time steps, Gautschi-KSS does not always exhibit second-order accuracy; this is due to the lack
of smoothness in the solution.

\begin{table}[ht]
\caption{Relative errors, execution times in seconds, and iteration counts in the solution of (\ref{eqKGEv}), 
(\ref{eqnumexpF}), (\ref{eqnumexpG}) with periodic boundary conditions on the domain
$(0,2\pi)\times(0,1)$, using the Gautschi-type method from \cite{grimm,hochlub} with
matrix function-vector products computed as in \cite{hochlubsel}
(labeled ``Gautschi-Krylov''), and the Gautschi-type method with matrix function-vector
products computed via KSS (labeled ``Gautschi-KSS'').
All methods use $N$ grid points, time step $\Delta t$, 
and central differencing in space.}
\label{tabprob12g}
\begin{center}
\footnotesize
\begin{tabular}{|c|r|r|r|r|r|r|} \hline
\multicolumn{2}{|c|}{} & \multicolumn{3}{|c|}{Gautschi-Krylov} & \multicolumn{2}{|c|}{Gautschi-KSS} \\ \hline
$N$ & $\Delta t$ & error & time & iter. & error & time \\ \hline
 & $\pi/64$ &   1.796e-04 & 0.013 & 8 &  1.933e-03 & 0.000 \\
 & $\pi/128$ &  4.309e-05 & 0.015 & 6 &  9.822e-05 & 0.001 \\
256 & $\pi/256$ &  1.036e-05 & 0.024 & 5 &  2.158e-05 & 0.001 \\
 & $\pi/512$ &  2.533e-06 & 0.033 & 4 &  5.267e-06 & 0.003 \\
 & $\pi/1024$ &  6.346e-07 & 0.055 & 4 &  1.304e-06 & 0.005 \\
 & $\pi/2048$ &  1.588e-07 & 0.095 & 3 &  3.254e-07 & 0.011 \\ \hline
 & $\pi/64$ &  1.616e-04 & 0.025 & 12 &  1.306e-03 & 0.001\\
 & $\pi/128$ &  4.434e-05 & 0.028 & 9 &  4.793e-04 & 0.001 \\
512 & $\pi/256$ &  1.068e-05 & 0.036 & 7 &  2.320e-05 & 0.002 \\
 & $\pi/512$ &  2.600e-06 & 0.059 & 6 &  5.339e-06 & 0.004 \\
 & $\pi/1024$ &  6.228e-07 & 0.090 & 5 &  1.300e-06 & 0.009 \\
 & $\pi/2048$ &  1.575e-07 & 0.148 & 4 &  3.228e-07 & 0.017 \\ \hline
 & $\pi/64$ &  1.614e-04 & 0.110 & 18 &  1.237e-03 & 0.001 \\
 & $\pi/128$ &  4.103e-05 & 0.073 & 12 &  3.384e-04 & 0.002 \\
1024 & $\pi/256$ &  1.066e-05 & 0.075 & 9 &  1.169e-04 & 0.004 \\
 & $\pi/512$ &  2.663e-06 & 0.089 & 7 &  5.598e-06 & 0.007 \\
 & $\pi/1024$ &  6.429e-07 & 0.138 & 6 &  1.314e-06 & 0.015 \\
 & $\pi/2048$ &  1.512e-07 & 0.228 & 5 &  3.238e-07 & 0.029 \\ \hline
\end{tabular}
\normalsize
\end{center}
\end{table}

\FloatBarrier

\section{Conclusion}

We have established an upper bound for the approximate solution operator of a second-order KSS method applied to the 1-D wave equation with bandlimited coefficients and periodic boundary conditions.
Unfortunately, the bound is not independent of the grid size, which indicates that the unconditional
stability proved for the heat equation for the same kind of spatial differential operator in \cite{paper2}
does not extend to the wave equation.  Numerical experiments support this assertion, while also suggesting that conditional stability may still hold.  Unlike the CFL condition, which relates the 
spatial grid mesh and time step to the {\em magnitude} of the wave speed, a stability condition
for a KSS method would likely relate the spatial grid mesh and time step to some measure of the
{\em variation} in the wave speed.

We have also proved that the same KSS method, applied to the wave equation with periodic boundary conditions, is convergent with spectral accuracy in space and second-order accuracy in time,
as well as unconditionally stable, in the case of a constant wave speed and a bandlimited reaction
term coefficient. This is the first result proving unconditional stability for a KSS method, for the wave
equation, that approximates the solution operator of the PDE using prescribed interpolation points, as opposed to nodes of Gauss quadrature rules.  Numerical experiments suggest that this unconditional stability also holds for related problems.

Furthermore, it has been demonstrated through numerical experiments, and then proved,
that the assumption of a bandlimited reaction term coefficient is not necessary for unconditional stability.
The proof of this result is the first stability analysis of a KSS method that does not require the coefficients of the spatial differential operator of the PDE to be either constant or bandlimited.  Future work will
extend this theory to other problems to which KSS methods have been applied.
Finally, it has been shown that KSS methods can be effective for nonlinear wave equations,
with an advantage in efficiency and scalability over other time-stepping methods that use Krylov subspace iterations, and that it is worthwhile to combine these approaches.  
Ongoing work involves combination of higher-order KSS methods and exponential
integrators \cite{hochost1,hochost2} to improve on the combination presented in \cite{kssepi}.

KSS methods, as presented in this paper and in \cite{paper8}, generalize the advantage of the Fourier spectral method for constant-coefficient linear PDEs--the ability to compute Fourier coefficients independently and with large time steps--to their variable-coefficient counterparts.  Although the discrete Fourier transform has served as an essential ingredient in KSS methods in these works, it is important to note that KSS methods and the DFT are not inextricably linked.  While the focus of this paper is mostly on problems in one space dimension with periodic boundary conditions, the main idea behind KSS methods--component-wise time stepping--can be employed effectively with any orthonormal basis (for example, a basis of modified sines, as used in \cite{nmpde}) for which transformation between physical space and frequency space can be carried out efficiently.  This allows for the development of KSS-like methods that use, for example, bases of Chebyshev polynomials or wavelets.
For problems on non-rectangular domains, combination with fictitious domain methods,
such as the Fourier continuation approach of \cite{bruno}, is worthy of investigation.
Another direction for future work is the addition of local time-stepping \cite{grote}, except
in frequency space rather than physical space, to handle the case of variable wave speed
by using smaller time steps for low-frequency components that are affected the most by such heterogeneity.



\section*{Acknowledgments}

The authors wish to thank the anonymous referees for their helpful feedback, which led to substantial improvement of the manuscript.


\bibliographystyle{plain}

\begin{thebibliography}{99}
\bibitem{Atk} K. Atkinson, {\em An Introduction to Numerical Analysis, 2nd Ed.} Wiley (1989).
\bibitem{bardostadmor} C. Bardos and E. Tadmor, ``Stability and spectral convergence of Fourier method for nonlinear problems. On the shortcomings of the 2/3 de-aliasing method'',  {\em Numerische Mathematik} {\bf 129} (2014), p. 749-782.
\bibitem{bruno} O. P. Bruno and P. Jagabandhu, ``Two-Dimensional Fourier Continuation and Applications'', {\em SIAM Journal on Scientific Computing} {\bf 44}(2) (2022), p. A964-A992.
\bibitem{kge} B. Bulbul, M. Sezer, and W. Greiner, {\em Relativistic Quantum Mechanics?Wave Equations}, Springer, Berlin, Germany, 3rd edition (2000).
\bibitem{kssepi} A. Cibotarica, J. V. Lambers and E. M. Palchak, ``Solution of Nonlinear Time-Dependent PDE Through Componentwise Approximation of Matrix Functions'', {\em Journal of Computational Physics} {\bf 321} (2016), p. 1120-1143.
\bibitem{dozier} H. Dozier, {\em Enhancement of Krylov Subspace Spectral Methods Through the Use of the Residual}, Ph.D. Dissertation (2019), {\tt https://aquila.usm.edu/dissertations/1658}
\bibitem{evans} L. C. Evans, {\em Partial Differential Equations}. American Mathematical Society (1998)
\bibitem{517txtbook} S. J. Farlow, \emph{Partial differential equations for scientists and engineers}. Dover Publications, Inc., New York (1993)
\bibitem{mmq} G. H. Golub and G. Meurant, ``Matrices, moments and quadrature''. In: Proceedings of the 15th Dundee Conference, June-July 1993. \emph{Longman Scientific and Technical} (1994)
\bibitem{gu} G. H. Golub and R. Underwood, ``The block Lanczos method for
computing eigenvalues'', {\em Mathematical Software III},  J. Rice Ed., (1977), p. 361-377.
\bibitem{GHT94}
J. Goodman, T. Hou, E. Tadmor, ``On the stability of the unsmoothed Fourier method for hyperbolic equations'', \emph{Numerische Mathematik} {\bf 67}(1) (1994), p. 93-129.
\bibitem{grimm} V. Grimm, ``On the Use of the Gautschi-Type Exponential Integrator for Wave Equations'',
{\em Numerical Mathematics and Advanced Applications}, A. B. de Castro, D. G\'{o}mez, P. Quintela and P. Salgado, eds., Springer (2006).
\bibitem{grote} M. J. Grote and T. Mitkova, ``Explicit local time-stepping methods for time-dependent wave propagation'', {\em Direct and Inverse Problems in Wave Propagation and Applications}, Ivan Graham, Ulrich Langer, Jens Melenk and Mourad Sini, eds., De Gruyter (2013), p. 187-218.
\bibitem{glk} P. Guidotti, J. V. Lambers and K. S\o lna, ``Analysis of the 1D Wave Equation in Inhomogeneous Media'', \emph{Numerical Functional Analysis and Optimization} {\bf 27} (2006), p. 25-55.
\bibitem{pdetxt} B. Gustafsson, H.-O. Kreiss and J. Oliger, \emph{Time dependent problems and difference methods}. Wiley, Amsterdam (1995) 
\bibitem{spectime} J. S. Hesthaven, S. Gottlieb and D. Gottlieb, {\em Spectral Methods for Time-Dependent Problems}. Cambridge University Press (2007)
\bibitem{hochlub} M. Hochbruck and C. Lubich, ``A Gautschi-type method for oscillatory second-order
differential equations'', {\em Numerische Mathematik} {\bf 83} (1999), p. 403-426.
\bibitem{hochlubsel} M. Hochbruck, M., C. Lubich and H. Selhofer, ``Exponential Integrators for Large Systems of Differential Equations'', {\em SIAM Journal on Scientific Computing} {\bf 19} (1998), p. 1552-1574.
\bibitem{hochost1} M. Hochbruck and A. Ostermann, ``Explicit exponential Runge-Kutta methods for semilinear parabolic problems'', {\em SIAM Journal on Numerical Analysis} {\bf 43} (2005), p. 1069-1090.
\bibitem{hochost2} M. Hochbruck and A. Ostermann, ``Exponential integrators of Rosenbrock-type'', {\em Oberwolfach Reports 3} (2006), p. 1107-1110.
\bibitem{paper17} J. V. Lambers, ``An Explicit, Stable, High-Order Spectral Method for the Wave Equation Based on Block Gaussian Quadrature", \emph{IAENG Journal of Applied Mathematics} \textbf{38} (2008), p. 233-248.
\bibitem{paper22} J. V. Lambers, ``Derivation of High-Order Spectral Methods for Time-Dependent PDE Using Modified Moments", \emph{Electronic Transactions on Numerical Analysis} \textbf{28} (2008), p. 114-135.
\bibitem{paper18} J. V. Lambers, ``Enhancement of Krylov Subspace Spectral Methods by Block Lanczos Iteration", \emph{Electronic Transactions on Numerical Analysis} \textbf{31} (2008), p. 86-109.
\bibitem{nmpde} J. V. Lambers and P. M. Jordan, ``On the application of a Krylov subspace spectral method to poroacoustic shocks in inhomogeneous gases'', \emph{Numerical Methods for Partial Differential Equations} {\bf 37}(6) (2021), p. 2955-2972. 
\bibitem{adapkrylov}
J. Niesen and W. M. Wright, ``Algorithm 919: A Krylov subspace algorithm for evaluating the $\varphi$-functions appearing in exponential integrators'', {\em ACM Transactions on Mathematical Software} {\bf 38}(3) (2012) p. 1-19.
\bibitem{ostermann22}
D. Phan and A. Ostermann, ``Exponential Integrators for Second-Order in Time Partial Differential
Equations'', {\em Journal of Scientific Computing} {\bf 93}:58 (2022).
\bibitem{paper8} E. M. Palchak, A. Cibotarica and J. V. Lambers, ``Solution of Time-Dependent PDE Through Rapid Estimation of Block Gaussian Quadrature Nodes", \emph{Linear Algebra and its Applications} \textbf{468} (2015), p. 233-259.
\bibitem{ilu} Y. Saad, {\em Iterative Methods for Sparse Linear Systems}, PWS Publishing Company (1996).
\bibitem{shampine} L. F. Shampine and M. W. Reichelt, ``The MATLAB ODE Suite'', 
{\em SIAM Journal on Scientific Computing} {\bf 18}(1) (1997), p. 1-22.
\bibitem{paper2} S. Sheikholeslami, J. V. Lambers and C. Walker, ``Convergence Analysis of Krylov Subspace Spectral Methods for Reaction-Diffusion Equations", \emph{Journal of Scientific Computing} \textbf{78}(3) (2019), p. 1768-1789. 
\end{thebibliography}

\end{document}